\documentclass[reqno]{amsart}

\usepackage{amssymb}
\usepackage{color}

\newcommand{\p}{\partial}
\newcommand{\R}{{\mathbb{R}}}
\newcommand{\N}{{\mathbb{N}}}
\newcommand{\supp}{{\mathrm{supp}\,}}
\renewcommand{\doteq}{{\mathrm{:=}}}

\renewcommand{\log}{\ln}
\newcommand{\<}{\langle}
\renewcommand{\>}{\rangle}

\newcommand{\lin}{{\mathrm{lin}}}
\newcommand{\n}[1]{\left\|#1\right\|_{L^\infty_tL^\infty_r}}

\newtheorem{theorem}{Theorem}

\newtheorem{prop}{Proposition}
\newtheorem{lemma}{Lemma}
\newtheorem{definition}{Definition}

\theoremstyle{remark}
\newtheorem{remark}{Remark}

\everymath{\displaystyle}

\title{From~$p_0(n)$ to~$p_0(n+2)$}
\begin{document}
\baselineskip13pt

\author{M. D'Abbicco, S. Lucente, M. Reissig}

\address{Marcello D'Abbicco, Departamento de Computa\c{c}\~{a}o e Matem\'atica, Universidade de S\~{a}o Paulo (USP), FFCLRP, Av. dos Bandeirantes, 3900, CEP 14040-901, Ribeir\~{a}o Preto - SP - Brasil}

\address{Sandra Lucente, Department of Mathematics, University of Bari, Via E. Orabona 4 - 70125 BARI - ITALY}

\address{Michael Reissig, Faculty for Mathematics and Computer Science, Technical University Bergakademie Freiberg, Pr\"uferstr.9 - 09596 FREIBERG - GERMANY}

\thanks{The first author is supported by Funda\c{c}\~{a}o de Amparo \`{a} Pesquisa do Estado de S\~{a}o Paulo, grants
2013/15140-2 and 2014/02713-7, JP - Programa Jovens Pesquisadores em Centros Emergentes, research project \emph{Decay estimates for semilinear hyperbolic equations}. The first two authors are member of the Gruppo Nazionale per l'Analisi Matematica, la Probabilit\`a e le loro Applicazioni (GNAMPA) of the Istituto Nazionale di Alta Matematica (INdAM). The second author has been supported by GNAMPA}

\begin{abstract}
In this note we study the global existence of small data solutions to the Cauchy problem for the semi-linear wave equation with a not \emph{effective
scale-invariant} damping term, namely
\[ v_{tt}-\triangle v + \frac2{1+t}\,v_t = |v|^p, \qquad v(0,x)=v_0(x),\quad v_t(0,x)=v_1(x), \]
where $p>1$, $n\ge 2$. We prove blow-up in finite time in the subcritical range~$p\in(1,p_2(n)]$ and an existence result for $p>p_2(n)$, $n=2,3$. In this
way we find the critical exponent for small data solutions to this problem. All these considerations lead to the conjecture $p_2(n)=p_0(n+2)$ for $n\ge2$,
where $p_0(n)$ is the Strauss exponent for the classical wave equation.
\end{abstract}

\maketitle

\section{Introduction}

In this paper we study the global existence (in time) of small data solutions to
\begin{equation}\label{eq:CP2}
\begin{cases}
v_{tt}-\triangle v + \frac2{1+t}\,v_t = |v|^p, & t\geq0, \quad x\in\R^n,\\
v(0,x)=v_0(x), & x\in\R^n,\\
v_t(0,x)=v_1(x), & x\in\R^n,
\end{cases}
\end{equation}
in space dimensions~$n=2,3$. The damping term of this model is not effective (see \cite{Wirth}). Nevertheless, there should be an improving influence on
the critical exponent $p_{crit}$ in comparison with Strauss exponent $p_0(n)$ for
\begin{equation}\label{eq:CP2000}
\begin{cases}
w_{tt}-\triangle w = |w|^p, & t\geq0, \quad x\in\R^n,\\
w(0,x)=w_0(x), & x\in\R^n,\\
w_t(0,x)=w_1(x), & x\in\R^n,
\end{cases}
\end{equation}
here $p_0(n)$ is the positive solution to \[  (n-1)\,p^2 - (n+1)\,p -2 = 0.\]
By \emph{critical exponent} $p_{crit}$ for ~\eqref{eq:CP2} or ~\eqref{eq:CP2000}
we mean, that for small initial data in a suitable space, there exist global (in time) solutions if~$p>p_{crit}$, and there exist suitable (even
small) data such that there exist no global (in time) solutions if~$p\in(1,p_{crit}]$.

It has been recently shown that the \emph{critical exponent} for models with effective dissipation, this means, $\mu$ is sufficiently large (see
\cite{Wirth})
\begin{equation}\label{eq:CP}
\begin{cases}
v_{tt}-\triangle v + \frac\mu{1+t}\,v_t = |v|^p, & t\geq0, \quad x\in\R^n,\\
v(0,x)=v_0(x), & x\in\R^n,\\
v_t(0,x)=v_1(x), & x\in\R^n,
\end{cases}
\end{equation}
is~$p_{crit}=1+2/n$ (see Section \ref{Sec.overview} for details). The exponent~$1+2/n$ is the same as for the semi-linear heat equation, and it is related
to the \emph{effectiveness} of the damping, i.e., the property of the damping term to make suitable linear estimates for the wave equation similar to the
ones for the corresponding heat equation~$\mu v_t-(1+t)\triangle v=0$ (in particular, the $L^1-L^p$ low frequencies estimates). We set
\begin{equation*}
p_\infty(n)=1+2/n,
\end{equation*}
where the index~$\infty$ means that~$\mu$ is sufficiently large.

On the contrary, it seems to be difficult to show that for small positive values of~$\mu$, for example, $\mu=2$ the \emph{critical exponent}~$p_\mu(n)$ is
strictly larger than $p_\infty$. Summarizing all these explanations one would expect
\[ p_\infty(n) \leq p_\mu(n) \leq p_0(n).\]
In this paper, we reach this aim by setting~$\mu=2$ in~\eqref{eq:CP}, and showing that
\begin{equation}\label{eq:p2}
p_2(n) \doteq \max \{ p_0(n+2), p_\infty(n)\} = \begin{cases}
3 & \text{if~$n=1$,}\\
p_0(n+2) & \text{if~$n=2,3$.}
\end{cases}
\end{equation}
We notice that~$p_\infty(2)=p_0(4)=2$ and $p_\infty(3)<p_0(5)$. Hence, for $n= 3$ and $\mu=2$ in ~\eqref{eq:CP}, we really feel the influence of the
non-effective dissipation term.

\smallskip
We prove the following results:
\begin{theorem}\label{th:mu2}
Assume that~$v\in\mathcal{C}^2([0,T)\times \R^n)$ is a solution to~\eqref{eq:CP2} with (even small) initial data $(v_0,v_1)\in \mathcal C_c^2(\R^n)\times
\mathcal C_c^1(\R^n)$ such that $v_1, v_0\geq0$, and $(v_0,v_1)\not \equiv (0,0)$. If~$p\in(1, p_2(n)]$, then~$T<\infty$.
\end{theorem}
Being the $1$-dimensional existence result already proved in~\cite{DA14+}, we prove the existence result for space dimension~$n=2$ and~$n=3$.
\begin{theorem}\label{Thm:ex}
Let $n=2$ and $p>2$. Let~$(\bar v_0,\bar v_1)\in \mathcal{C}^2_c(\R^2)\times \mathcal{C}^1_c(\R^2)$. Then there exists~$\varepsilon_0>0$ such that for
any~$\varepsilon\in(0,\varepsilon_0)$, if~$v_0=\varepsilon\bar v_0$ and~$v_1=\varepsilon\bar v_1$, then the Cauchy problem~\eqref{eq:CP2} admits a unique
global (in time) small data solution $v\in \mathcal{C}([0,\infty),H^2)\cap \mathcal{C}^1([0,\infty),H^1)\cap \mathcal{C}^2([0,\infty),L^2)$.
\end{theorem}
\begin{theorem}\label{Thm:3}
Let $n=3$ and $p>p_0(5)$. Let~$(\bar v_0,\bar v_1)\in \mathcal{C}^2_c(\R^3)\times \mathcal{C}^1_c(\R^3)$, be radial. Then there exists~$\varepsilon_0>0$
such that for any~$\varepsilon\in(0,\varepsilon_0)$, if~$v_0=\varepsilon\bar v_0$ and~$v_1=\varepsilon\bar v_1$, then the Cauchy problem~\eqref{eq:CP2}
admits a unique global (in time) small data radial solution $v\in \mathcal{C}([0,\infty)\times\R^3)\cap \mathcal{C}^2([0,\infty)\times
(\R^3\setminus\{0\}))$.
\end{theorem}
For the sake of brevity, we use the notation
\begin{equation*}
\<y\>=1+|y| \text{ for any } y\in\R^n.
\end{equation*}
To prove our results we perform the change of variable~$u(t,x)=\<t\>v(t,x)$. So, problem~\eqref{eq:CP2} becomes
\begin{equation}\label{eq:CPnew}
\begin{cases}
u_{tt}-\triangle u = \<t\>^{-(p-1)}|u|^p, & t\geq0, \quad x\in\R^n,\\
u(0,x)=u_0(x), & x\in\R^n,\\
u_t(0,x)=u_1(x), & x\in\R^n
\end{cases}
\end{equation}
with~$u_0=v_0+v_1$ and~$u_1=v_1$. This means we are dealing with a semi-linear wave equations with a time-dependent coefficient in the nonlinearity. For
proving Theorem \ref{th:mu2} we will extend to this equation the classical blow-up technique due to R.T. Glassey. For Theorem \ref{Thm:ex} we use
Klainerman vector fields. Due to the lack of regularity of the nonlinear term, for~$p\in (p_0(5),2)$, the proof of Theorem \ref{Thm:3}
requires a different idea and we shall restrict to radial solutions. We will establish an appropriate version of the pointwise estimates for the wave equation. \\
By the aid of these estimates, in this latter case, we will also find a decay behaviour for the solution to \eqref{eq:CP2} which is the same as for
solutions to the $(n+2)$-dimensional wave-equations. For details, see Theorem~\ref{thm:final} and Remark~\ref{thm:final}. We expect that the technique of
the pointwise estimates could be applied to prove the existence for~$p>p_0(n+2)$ with $n\ge 4$. Consequently, the improving influence of the dissipation
term on the one hand and the non-effective behaviour on the other hand  can be expressed by a shift of Strauss exponent $p_0(n)$ to $p_0(n+2)$ as in the title of this paper.

\section{An overview of some existing results}\label{Sec.overview}
\subsection{Wave model}
For the Cauchy problem ~\eqref{eq:CP2000} it is well-known that the \emph{critical exponent} for the existence of global (in time) small data solutions
is~$p_0(n)$. More precisely, if~$1<p\leq p_0(n)$, then solutions to~\eqref{eq:CP2000} blow-up in finite time for a suitable choice of initial data (see
\cite{G}, \cite{J}, \cite{JZ}, \cite{Sc}, \cite{Si}, \cite{YZ06}), whereas for $p\in(p_0(n),(n+3)/(n-1)]$ a unique global (in time) small data solution
exists (see \cite{GLS}, \cite{G2}, \cite{J}, \cite{T}, \cite{Z}). In space dimension~$n=1$, solutions to~\eqref{eq:CP2000} blow-up in finite time for
any~$p>1$, hence, we put~$p_0(1)=\infty$ (see~\cite{G}).
\subsection{Scale-invariant damped wave models}
Known results on the global existence of small data solutions to~\eqref{eq:CP} can be summarized as follows:
\begin{itemize}
\item Non-existence of weak solutions for $\mu>1$ and $p\le 1+2/n$, provided $\int_{\mathbb{R}^n} \big(u_0+(\mu-1)^{-1}u_1\big) \,d\,x>0$ (see
Theorem 1.1 and Example 3.1 in~\cite{DAL13}).
\item Non-existence of weak solutions for $\mu\in(0,1]$ and $p\le 1+2/(n-1+\mu)$, provided $\int_{\mathbb{R}^n} u_1 \,d\,x>0$ (see Theorem 1.4 in~\cite{W}).
\item According to Theorems~2 and~3 in~\cite{DA14+} global (in time) existence of energy solutions for small data if $p>1+2/n$ and
\begin{itemize}
\item if $n=1$ and $\mu \ge 5/3$,
\item if $n=2$ and $\mu \ge 3$,
\item for any~$n\geq3$ if $\mu\ge n+2$.
\end{itemize}
\end{itemize}
\subsection{Useful transformations}
The equation in~\eqref{eq:CP} has many interesting properties. In particular, if~$\mu\in\R$, by the change of variables
\begin{equation}\label{eq:musharp}
v^\sharp(t,x)=\<t\>^{\mu-1}\,v(t,x)
\end{equation}
one sees that $v$ solves~\eqref{eq:CP} if, and only if, $v^\sharp(t,x)$ solves
\begin{equation}\label{eq:usharp}
\left\{
\begin{array}{l}
v_{tt}^\sharp- \triangle v^\sharp + \frac{\mu^\sharp}{\<t\>}\,v_t^\sharp = \<t\>^{(\mu^\sharp-1)(p-1)} |v^\sharp|^p, \\
v^\sharp(0,x)=v_0(x),\\
v_t^\sharp(0,x)=v_1(x)+(1-\mu^\sharp)v_0(x)
\end{array}
\right.
\end{equation}
with $\mu^\sharp=2-\mu\,.$\\
If~$\mu\in(-\infty,1)$ in~\eqref{eq:CP}, by introducing the change of variable~$\tilde{v}(t,x)=v(\Lambda(t)-1,x)$, where
\begin{equation}\label{eq:Lambdapol}
\Lambda(t)\doteq \frac{\<t\>^{\ell+1}}{\ell+1}, \qquad \text{and} \quad \ell=\frac\mu{1-\mu},
\end{equation}
the Cauchy problem~\eqref{eq:CP} becomes a Cauchy problem for a semi-linear free wave equation with polynomial propagation speed
\begin{equation}\label{eq:wave01}
\left\{
\begin{array}{l}
\tilde{v}_{tt}-\<t\>^{2\ell}\triangle \tilde{v}= \<t\>^{2\ell} |\tilde{v}|^p,\\
\tilde{v}(\bar t,x)=v_0(x),\\
\tilde{v}_t(\bar t,x)=(1-\mu)^{-\mu}v_1(x),
\end{array}
\right.
\end{equation}
where $\bar t=(1-\mu)^{-(1-\mu)}-1$. We notice that:
\begin{itemize}
\item $\ell>0$ if, and only if,~$\mu\in(0,1)$. On the other hand, $\ell\in(-1,0)$ if~$\mu\in(-\infty,0)$;
\item $\bar t\in(0,e^{\frac1e}-1]$ if~$\mu\in(0,1)$ and~$\bar t\to0$ as~$\mu\to0$ and as~$\mu\to1$;
\item $\bar t\in(-1,0)$ if~$\mu\in(-\infty,0)$.
\end{itemize}

Similarly, by virtue of \eqref{eq:musharp}, \eqref{eq:usharp} and \eqref{eq:Lambdapol}, if~$\mu>1$, the Cauchy problem \eqref{eq:CP} becomes
\begin{equation}\label{eq:wave12}
\left\{
\begin{array}{l}
\tilde{v}^\sharp_{tt}-\<t\>^{2\ell^\sharp}\triangle \tilde{v}^\sharp= c_\mu\,\<t\>^{2\ell^\sharp-(p-1)} |\tilde{v}^\sharp|^p,\\
\tilde{v}^\sharp(t^\sharp,x)= v_0(x),\\
\tilde{v}^\sharp_t(t^\sharp,x)=(\mu-1)^{-(2-\mu)}( v_1(x)+(\mu-1) v_0(x)),
\end{array}
\right.
\end{equation}
where~$\ell^\sharp=(2-\mu)/(\mu-1)$, $t^\sharp=(\mu-1)^{-(\mu-1)}-1$, and $c_\mu=(\mu-1)^{(\mu-1)(p-1)}$.

On the other hand, if~$\mu=1$, by setting~$\Lambda(t)=e^t$, the Cauchy problem~\eqref{eq:CP} becomes
\begin{equation}
\label{eq:expwave}
\left\{
\begin{array}{l}
\tilde{v}_{tt}-e^{2t}\triangle \tilde{v}= e^{2t}\,|\tilde{v}|^p,\\
\tilde{v}(0,x)=v_0(x), \\
\tilde{v}_t(0,x)=v_1(x).
\end{array}
\right.
\end{equation}

By means of all these transformations, following the reasoning in Example 4.4 in~\cite{DAL13}, we can obtain as in ~\cite {W} the non-existence of global (in time)
weak solutions to~\eqref{eq:CP} for $\mu\in(0,1)$ and
\[ p\le 1+\frac{2(\ell+1)}{n(\ell+1)-1}=1+\frac{2}{n-1+\mu}. \]
\subsection{Blow-up dynamics}
Since in~\cite{DAL13, W} the test function method is employed, the blow-up dynamic remains unknown. However, one can apply an argument similar to those
developed in \cite{G}, to~\eqref{eq:wave01}, \eqref{eq:wave12} and~\eqref{eq:expwave}, obtaining that all the $L^q$ norms of local solutions blow-up in
finite time. Indeed, in Example 2a in~\cite{Yag} the author gives sufficient conditions on
$$
u_{tt}-a^2(t)\Delta u=m(t)|u|^p,
$$
which guarantee that $\lim_{t\to T} \|u(t)\|_q=+\infty$ for any $1\le q\le +\infty$, where~$T$ is the maximal
existence time for a smooth solution with nonnegative, compactly supported, initial data. See also \cite{Ga} for the 1-dimensional case.
%This result can be applied both for $a(t)=(1+t)^{\ell}$ and $m(t)=(1+t)^{2\ell}$ for $\ell>0$ and also for $a(t)=e^t$ and $m(t)=e^{2t}$.
%Summarizing we have
By means of~\eqref{eq:wave01}, \eqref{eq:wave12} and~\eqref{eq:expwave} from these results one can deduce the blow-up in finite time for~\eqref{eq:CP}
\begin{itemize}
\item if $\mu\in(0,1)$ and $p<1+\frac{2}{n-1+\mu}$,
\item if $\mu=1$ and $p\le p_\infty$,
\item if $\mu\in (1,2]$ and $p< p_\infty$.
\end{itemize}
We notice that blow-up in finite time is proved for the limit case $p=1+2/(n-[1-\mu]^+)$ only for $\mu=1$, while non-existence of weak solutions
for~$\mu\in(0,1)\cup (1,2]$ is also known for $p=1+2/(n-[1-\mu]^+)$. Up to our knowledge we have no other information from literature about existence or
non-existence for \eqref{eq:CP}. In particular, the blow-up dynamic is unknown for $\mu>2$.

After this discussion, it was natural to ask if the blow-up exponent~$p_\infty(n)=1+2/n$ could be improved for some~$\mu\in[1,5/3)$ if~$n=1$, for
some~$\mu\in[1,3)$ if~$n=2$, or for some~$\mu\in[1,n+2)$ if~$n\geq3$. On the other hand, one may ask if  a counterpart result of global (in time) existence
can be proved. Theorems \ref{th:mu2}, \ref{Thm:ex}, \ref{Thm:3} give a positive answer to these questions in the special case $\mu=2$. These results may give precious hints
about the general case of small $\mu$.
\subsection{Space-dependent damping term}
For the sake of completeness, we remark that the case of wave equation with space-dependent damping
\begin{equation}
\label{eq:CPx}
\begin{cases}
v_{tt}-\triangle v + \mu\<x\>^{-\alpha}\,v_t = |v|^p, & t\geq0, \quad x\in\R^n,\\
v(0,x)=v_0(x), & x\in\R^n,\\
v_t(0,x)=v_1(x), & x\in\R^n,
\end{cases}
\end{equation}
where~$\mu>0$ and~$\alpha\in(0,1]$, is also particularly difficult when~$\alpha=1$. On the one hand, in~\cite{ITY09}, the authors proved that the critical
exponent for the existence of global (in time) small data solutions is~$1+2/(n-\alpha)$ if~$\alpha\in(0,1)$. On the other hand, in~\cite{ITY13} they proved for $\alpha=1$
that the estimates for the energy of solutions to the linear model of~\eqref{eq:CPx} show a decay rate which depends on~$\mu$ for~$\mu < n$. This property hints to a
$\mu$-depending critical exponent for~\eqref{eq:CPx} for small~$\mu$.

To complete our overview, we mention that the critical exponent for the wave equation with time-dependent damping~$\mu\<t\>^\kappa u_t$ is~$1+2/n$ if~$\kappa\in(-1,1)$ (see~\cite{DALR13,LNZ12, N11}), whereas global existence of small data solutions for~$p>1+2/(n-\alpha)$ for the wave equation with damping~$\mu\<x\>^{-\alpha}\<t\>^{-\beta}$, if~$\alpha,\beta>0$ and~$\alpha+\beta<1$ has been derived in~\cite{W12}.

%%%%%%%%%%%%%%%%%%%%%%%%%%%%%%%%%%%%%%%%%%%%%%%%%%%%%%%%%%

\section{Proof of Theorem~\ref{th:mu2}}%\label{sec:sub}

Let us remind the blow-up dynamics for ordinary differential inequalities with polynomial nonlinearity. This result will play a fundamental role in our
approach.
\begin{lemma}\label{Lem:1YZ}
Let~$p>1$, $q\in\R$. Let~$F\in\mathcal{C}^2([0,T))$, positive, satisfying
\begin{equation}\label{eq:F2est}
\ddot{F}(t) \geq K_1 (t+R)^{-q} (F(t))^p \qquad \text{for any~$t\in[T_1,T)$}
\end{equation}
for some~$K_1,R>0$, and~$T_1\in[0,T)$. If
\begin{equation}\label{eq:Fest}
F(t)\geq K_0 (t+R)^a\qquad \text{for any~$t\in[T_0,T)$,}
\end{equation}
for some $a\ge 1$ satisfying $a>(q-2)/(p-1)$, and for some~$K_0>0$, $T_0\in[0,T)$, then~$T<\infty$.

Moreover, let $q\ge p+1$ in~\eqref{eq:F2est}. Then there exists a constant~$K_0=K_0(K_1)>0$ such that, if~\eqref{eq:Fest} holds with~$a=(q-2)/(p-1)$ for
some~$T_0\in[0,T)$, then~$T<\infty$.
\end{lemma}
\begin{proof} The case $a>(q-2)/(p-1)$ corresponds to Lemma~4 in \cite{Si}. Let $a=(q-2)/(p-1)$. Following Lemma~2.1 in~\cite{YZ06}, our problem reduces to find $K_0$ such that \eqref{eq:Fest} holds and the function
$G(s)=(T_0+1)^{-a}F((T_0+1)s+1)$ blows up. One has
\begin{equation*}
\begin{array}{l}
\ddot{G}(s) \geq K_1 \<s\>^{-q} (G(s))^p,\\
G(s)\geq \tilde{K}_0 \<s\>^{\frac{q-2}{p-1}},
\end{array}
\end{equation*}
respectively, with $\tilde{K}_0=K_0 \min\{1;(1+R)/(1+T_0)\}$. Eventually, with a larger constant $K_0$ it follows that $\dot G$ is positive, so that from $ \ddot{G}(s) \geq K_1 \tilde{K}_0^{p-1} \<s\>^{-2} (G(s))$ one has $G(s)\ge
\<s\>^{\tilde{K}_0^{p-1}K_2}$. For large $K_0$, the exponent $a\doteq \tilde{K}_0^{p-1}K_1$ satisfies~$a>(q-2)/(p-1)$, and we may conclude the proof. These ideas are
contained in \cite{G}.
\end{proof}
Transforming problem~\eqref{eq:CP2} into~\eqref{eq:CPnew} the statement follows as a consequence of the next proposition. Here we follow~\cite{YZ06},
taking into account of the time-dependence of the nonlinear term.
\begin{prop}\label{prop:non}
Let~$f\in \mathcal C^2(\R^n)$ and $g\in\mathcal C^1(\R^n)$ be nonnegative, compactly supported, such that $f+g \not\equiv 0$. Assume
that~$u\in\mathcal{C}^2([0,T)\times \R^n)$ is the maximal (with respect to the time interval) solution to
\begin{equation}\label{eq:wave20}
\left\{
\begin{array}{l}
u_{tt}-\triangle u=\<t\>^{-(p-1)} |u|^p, \\
u(0,x)=f(x), \\
u_t(0,x)=g(x).
\end{array}
\right.
\end{equation}
If~$p\le p_2(n)$, with~$p_2(n)$ as in~\eqref{eq:p2}, then~$T<\infty$.
\end{prop}
%
%%%%%%%%%%%%%%%%%%%%%%%%%%%%%%%%%%%%
Let~$R>0$ be such that~$\supp f, \supp g\subset B(R)$. Therefore, $\supp u(t,\cdot)\subset B(R+t)$. Without loss of generality we assume~$R=1$. Let us
define
\[
F(t)\doteq \int_{\R^n} u(t,x)\,dx.
\]
Thanks to the finite speed of propagation of~$u$ and by H\"older's inequality
\begin{align} \label{eq:F2}
\begin{array}{cc} \ddot F(t)
    = \<t\>^{-(p-1)}\,\int_{\R^n}|u(t,x)|^p\,dx \\
    = \<t\>^{-(p-1)}\,\int_{B(\<t\>)}|u(t,x)|^p\,dx \gtrsim \<t\>^{-(n+1)(p-1)}\,|F(t)|^p.
\end{array}
\end{align}
In order to apply Lemma~\ref{Lem:1YZ} we need to establish that $F(t)$ is positive. For this reason we consider the functions
\[
\phi_1(x)\doteq \int_{S^{n-1}} e^{x\cdot \omega}\,d\omega, \qquad  \psi_1(t,x)\doteq\phi_1(x)e^{-t}
\]
and
\[
F_1(t)\doteq \int_{\R^n} u(t,x)\psi_1(t,x)\,dx.
\]
It follows that
\[
\ddot F(t) \gtrsim \<t\>^{-(p-1)}\,|F_1(t)|^p\,\Bigl(\int_{B(1+t)} (\psi_1(t,x))^{\frac{p}{p-1}}\,dx\Bigr)^{-(p-1)}.
\]
Let us estimate the last integral. Recalling that $\psi_1(t,x) = e^{-t}\phi_1(x)$ we see that
\[
\int_{B(K)} (\psi_1(t,x))^{\frac{p}{p-1}}\,dx\le C(K,A,p) \<t\>^{-A}
\]
for any fixed $K<1+t$ and $A>0$.
By using
\[
\phi_1(x)\lesssim |x|^{-\frac{n-1}2}\,e^{|x|} \quad \text{as~$|x|\to\infty$}
\]
(see \cite{CH}, pages 184,185) we get for large $t$ and $K$ the estimate
\[
\int_{B(1+t)\setminus B(K)} (\psi_1(t,x))^{\frac{p}{p-1}}\,dx\lesssim \int_K^{t+1} \<\rho\>^{n-1-\frac{(n-1)p}{2(p-1)}}e^{\frac{p}{p-1}(\rho-t)}\,d\rho.
\]
Putting
\[
\alpha\doteq n-1-(n-1)p/(2(p-1)),
\]
we have
\[
\int_K^{t+1} \<\rho\>^{\alpha}e^{\frac{p}{p-1}(\rho-t)}\,d\rho\le \frac{p-1}{p} e^{\frac{p}{p-1}}(2+t)^\alpha
-\frac{\alpha(p-1)}{p}\int_K^{t+1} e^{\frac{p}{p-1}(\rho-t)}\<\rho\>^{\alpha-1}\,d\rho.
\]
If~$\alpha\ge 0$, i.e. $p\ge 2$, we may immediately conclude
\begin{equation}\label{eq:inter}
\int_K^{t+1} \<\rho\>^{\alpha}e^{\frac{p}{p-1}(\rho-t)}\,d\rho\lesssim \<t\>^\alpha.
\end{equation}
The same estimate holds if $\alpha<0$, i.e.~$p\in(1,2)$, since we may write
\[
\left(1+\frac{\alpha(p-1)}{p(1+K)}\right) \int_K^{t+1} \<\rho\>^{\alpha}e^{\frac{p}{p-1}(\rho-t)}\,d\rho\lesssim \<t\>^{\alpha}
\]
and for large $K$ and $t$ we turn to \eqref{eq:inter}.  As a conclusion
\[ \ddot F(t)  \gtrsim \<t\>^{-n(p-1)+(n-1)\frac{p}2}\,|F_1(t)|^p . \]
To estimate~$|F_1(t)|^p$ the sign of the nonlinear term comes into play. More precisely, the following result holds for any smooth solution to
$u_{tt}-\Delta u=G(t,x,u)$ with positive~$G$.
\begin{lemma}\label{Lem:2YZ}\emph{[Lemma~2.2 in~\cite{YZ06}]}
There exists a positive constant~$t_0$ such that it holds
\begin{equation}\label{eq:lem2in}
F_1(t) \gtrsim \frac{1}{2}(1-e^{-2t})\int_{\R^n}(f(x)+g(x))\phi_1(x)\,dx+e^{-2t}\int_{\R^n}f(x)\phi_1(x)\,dx
\end{equation}
for~$t\geq t_0$.
\end{lemma}
In particular, due to our assumption on $f$ and $g$, it holds $F_1(t)>c>0$.
Therefore, we proved
\begin{equation}\label{eq:F2bis}
\ddot F(t) \gtrsim \<t\>^{-n(p-1)+(n-1)\frac{p}2} = \<t\>^{-\frac{n+1}2\,p+n}.
\end{equation}
Integrating twice we obtain
\begin{equation}\label{eq:F}
F(t) \gtrsim \<t\>^{\max\left\{-\frac{n+1}2\,p+n+2,1\right\}} + t \dot{F}(0) + F(0) \gtrsim \<t\>^{\max\left\{-\frac{n+1}2\,p+n+2,1\right\}},
\end{equation}
since~$\dot{F}(0)\ge 0$ and $F(0)\ge 0$.
\subsection{The subcritical case}
Recalling \eqref{eq:F2} we may apply the first part of Lemma~\ref{Lem:1YZ} once we have one of the following conditions:
\begin{align}
\label{eq:Straussbirth}
-\frac{n+1}2\,p+n+2 & > \frac{(n+1)(p-1)-2}{p-1}\,, \\
\label{eq:Fujitabirth}
1 & > \frac{(n+1)(p-1)-2}{p-1}\,.
\end{align}
Condition~\eqref{eq:Straussbirth} corresponds to~$p<p_0(n+2)$, whereas condition~\eqref{eq:Fujitabirth} corresponds to~$p<p_\infty(n)$, hence, we derived
$p<\max\{p_0(n+2), p_\infty(n)\}$.

\subsection{Critical case in 1d} First, let~$n=1$ and~$p=3$. By~\eqref{eq:F2} it follows~\eqref{eq:F2est} with~$q=4$. Setting~\eqref{eq:F}
into~\eqref{eq:F2} leads to
\[ \ddot F(t) \gtrsim \<t\>^{-4}\,F(t)^3\gtrsim \<t\>^{-1}. \]
Integrating twice implies $F(t)\gtrsim \<t\>\,\log \<t\>$. Therefore, for any~$K_0>0$ there exists~$T_0>0$ such that~\eqref{eq:Fest} holds with~$a=1$. The
proof follows from Lemma~\ref{Lem:1YZ}.

\subsection{Critical case in 2d}  By~\eqref{eq:F2} and~\eqref{eq:F} we have again  $\ddot F(t) \gtrsim  \<t\>^{-1}$. Consequently, the conclusion follows.

\subsection{Critical case in $n$ dimensions with $n\ge 3$} We notice that $p_2(n)=p_0(n+2) < 2$. Due to the lack of $C^2$ regularity of solutions  we shall prove a blow-up behaviour for the spherical mean of $u$,
that is, for
\[ \tilde u(t,r)=\frac{1}{\omega_n}\int_{|\omega|=1} u(t,r\omega)\,dS_\omega.\]
This mean satisfies the differential inequality (see \cite{JZ})
\[ \tilde u_{tt}-\Delta \tilde u\ge \<t\>^{-(p-1)}|\tilde u|^p.\]
We can assume that $u$ is radial. Following \cite{YZ06} we consider the Radon transform of~$u$ on the hyper-planes orthogonal to a fixed $\omega\in \R^n$:
\[ Ru(t,\rho) \doteq \int_{x\cdot w=\rho} u(t,x)\,dS_x\,, \]
where~$dS_x$ is the Lebesgue measure of $\{ x: \ x\cdot w=\rho \}$. One can see that $Ru$ is independent of~$w$ and that
\begin{equation}\label{eq:Ru}
Ru(t,\rho)= c_n \int_{|\rho|}^\infty u(t,r)\,(r^2-\rho^2)^{\frac{n-3}2}\,r\,dr\,.
\end{equation}
We will assume that~$\rho\geq0$. Since~$Ru$ satisfies
\[ \partial_t^2 Ru - \partial_\rho^2 Ru = \<t\>^{-(p-1)}\,R(|u|^p) \]
and~$f\geq0$, $g\geq0$, it follows
\[ Ru(t,\rho) \geq \frac12\,\int_0^t \<s\>^{-(p-1)}\, \int_{\rho-(t-s)}^{\rho+(t-s)} R(|u|^p)(s,\rho_1) \,d\rho_1\,ds\,. \]
Since~$\supp R(|u|^p)(s,\cdot) \subset B(s+1)$, following~\cite{YZ06} we may estimate
\begin{align*}
Ru(t,\rho)
    & \geq \frac12\,\int_0^{\frac{t-\rho-1}{2}} \<s\>^{-(p-1)}\, \int_{\R}R(|u|^p)(s,\rho_1) \,d\rho_1\,ds \\
    & = \frac12\,\int_0^{\frac{t-\rho-1}{2}} \<s\>^{-(p-1)}\, \int_{\R^n} |u(s,y)|^p\,dy\,ds=
     \frac12\,\int_0^{\frac{t-\rho-1}{2}} \ddot F(s)\,ds.
\end{align*}
Recalling~\eqref{eq:F2bis} we get
\[ Ru(t,\rho) \geq \frac12\,\int_0^{\frac{t-\rho-1}{2}} \<s\>^{-\frac{n+1}2\,p+n}\,ds. \]
Since~$p=p_2(n)\leq 2$ it holds
\begin{equation}
\label{eq:Ruest}
Ru(t,\rho)\gtrsim %\begin{cases}
(1+t-\rho)^{-\frac{n+1}2\,p+n+1}. %& \text{if~$n\ge 3$,} \\
\end{equation}
Coming back to~\eqref{eq:Ru} and recalling that~$\supp u(t,\cdot)\subset B(1+t)$, since~$r+\rho\leq 2r$ in the integral, we may estimate
\begin{equation}\label{eq:Rufund}
Ru(t,\rho) = c_n \int_\rho^{1+t} u(t,r)\,r(r+\rho)^{\frac{n-3}2}\,(r-\rho)^{\frac{n-3}2}\,dr \leq c_n 2^{\frac{[n-3]^+}2} \int_\rho^{1+t} u(t,r)\,r^{\frac{n-1}2}\,(r-\rho)^{\frac{n-3}2}\,dr\,.
\end{equation}
The operator~$T: L^p(\R)\to L^p(\R)$ defined by
\[ Tf (\tau) \doteq \frac1{|1+t-\tau|^{\frac{n-1}2}}\int_\tau^{1+t} f(r)\,|r-\tau|^{\frac{n-3}2}\,dr  \qquad \text{for any~$\tau\in\R$,} \]
is bounded. Therefore, if we put
\[ f(r)=\begin{cases}
|u(t,r)|\,r^{\frac{n-1}p} & \text{if~$r\geq0$,}\\
0 & \text{if~$r\leq0$,}
\end{cases}\]
so that~$f(r)^p=|u(t,r)|^p\,r^{n-1}$ for~$r\geq0$, then we get
\[ \int_0^{1+t} \bigg(\frac1{|1+t-\rho|^{\frac{n-1}2}}\int_\rho^{1+t} |u(t,r)|\,r^{\frac{n-1}p}\,(r-\rho)^{\frac{n-3}2}\,dr\bigg)^p d\rho \lesssim \int_0^\infty |u(t,r)|^p\,r^{n-1}\,dr = C\int_{\R^n}|u(t,x)|^p\,dx \,. \]
Due to~$p\leq2$ and~$r\geq\rho$ it holds $r^{\frac{n-1}p}\geq r^{\frac{n-1}2}\rho^{(n-1)\left(\frac1p-\frac12\right)}$, so that, by~\eqref{eq:Rufund},
we conclude %as in~(2.21) in~\cite{YZ06},
\[ \int_0^{1+t} \frac{(Ru(t,\rho))^p}{|1+t-\rho|^{\frac{n-1}2\,p}}\,\rho^{(n-1)-(n-1)\frac{p}2} d\rho \lesssim \int_{\R^n}|u(t,x)|^p\,dx\,. \]
Thanks to~\eqref{eq:Ruest} we get
\[ \int_{\R^n}|u(t,x)|^p\,dx \gtrsim \int_0^{1+t} (1+t-\rho)^{-\frac{n+1}2\,p^2+\frac{n+3}2\,p}\,\rho^{(n-1)-(n-1)\frac{p}2} d\rho\,. \]
Recalling that~$p=p_2(n)$ we may use
\[ \frac{n+1}2\,p^2-\frac{n+3}2\,p = 1, \]
and obtain
\[ \int_{\R^n}|u(t,x)|^p\,dx \gtrsim \int_0^{1+t} (1+t-\rho)^{-1}\,\rho^{(n-1)-(n-1)\frac{p}2} d\rho \gtrsim \<t\>^{(n-1)-(n-1)\frac{p}2}\ln \<t\> . \]
%
%\footnote{L'idea \`e restringere l'intervallo di integrazione a $[t/2-1, t-2]$, cos\`i da avere che il numeratore si comporta come la potenza sostituendo~$t$ a ~$rho$ e integrando l'altro pezzo (occhio al segno che inverte l'integrale) viene $\log t$.}
%
Thus,
\[ \ddot F(t)\gtrsim \<t\>^{-(p-1)}\int_{\R^n}|u|^p\,dx \gtrsim \<t\>^{(n-1)-(n+1)\frac{p}2+1}\ln \<t\>\,, \]
%
%Since~$(n-1)-(n+1)p/2+1\geq0$ for~$p=p_2(n)\leq2$,
hence,
\[ F(t)\gtrsim \<t\>^{(n+1)-(n-1)\frac{p}2}\ln \<t\>\,. \]
Similarly to the case~$n=1$ the end of the proof follows by Lemma~\ref{Lem:1YZ}. %, since~$\log (1+t)\to\infty$ as~$t\to\infty$.

%%%%%%%%%%%%%%%%%%%%%%%%%%%%%%%%%%%%%%%%%%%%%

\section{Proof of Theorem~\ref{Thm:ex}}
Let ~$p,q \in[1,\infty]$. As in~\cite{LX} we define
\[ \|f\|_{(p,q)} \doteq \| f(r \omega) \, r^{\frac{n-1}p} \|_{L_r^p([0,\infty),L_\omega^q(S^{n-1}))}\,. \]
It holds~$\|f\|_{(p,p)}=\|f\|_{L^p}$ and H\"older's inequality
\begin{equation}\label{eq:Holder}
\|f_1\,f_2\|_{(p,q)} \lesssim \|f_1\|_{(p_1,q_1)}\,\|f_2\|_{(p_2,q_2)}
\end{equation}
is valid if
\[ \frac{1}{p}=\frac{1}{p_1}+\frac{1}{p_2}\le 1, \qquad \frac{1}{q}=\frac{1}{q_1}+\frac{1}{q_2}\le 1.\]
Moreover, since~$S^{n-1}$ is compact, it holds
\begin{equation}\label{eq:emb}
\|f\|_{(p,q_1)} \lesssim \|f\|_{(p,q_2)}\qquad \text{for any~$q_2\ge q_1$.}
\end{equation}
Let $i,j=1,2$ with $i\not=j$, we introduce the vector fields
\begin{gather*}
\Gamma = (D, L_0, L_j, \Omega_{ij}), \qquad D=(\partial_t,\partial_j), \quad L_0= \<t\>\p_t+x\cdot\nabla, \\
L_j= \<t\>\p_j + x_j\p_t, \quad \Omega_{ij} = x_i\p_j-x_j\p_i,
\end{gather*}
and the norms
\begin{equation*}
\begin{array}{l}
 \|f\|_{\Gamma,s,(p,q)} = \sum_{|\alpha|\leq s} \|\Gamma^\alpha f\|_{(p,q)}\,,\\
 \|f\|_{\Gamma,s,p}=\|f\|_{\Gamma,s,(p,p)}\,,\\
  \|f\|_{\Gamma,s,\infty}=\sum_{|\alpha|\leq s} \|\Gamma^\alpha f\|_\infty\,.
  \end{array}
\end{equation*}
To a given $\alpha$ the following relations hold with suitable $a_\beta$ and $b_\beta$:
\begin{align}
\label{eq:comLG}
[\square, \Gamma^\alpha]
    & = \sum_{|\beta|\leq |\alpha|-1} a_\beta\,\Gamma^\beta\square, \\
\label{eq:comDG}
[D, \Gamma^\alpha]
    & = \sum_{|\beta|\leq |\alpha|-1} b_\beta\,\Gamma^\beta D.
\end{align}
By using arguments from~\cite{Kl} one has the following Sobolev-type inequalities in these generalized Sobolev spaces:
\begin{eqnarray}
\label{eq:Kl1} &&\|w(t,\cdot)\|_{\infty} \lesssim \<t\>^{-\frac{n-1}2}\,\|w(t,\cdot)\|_{\Gamma,s,2} \quad \text{if~$s>n/2$,}
\\
\label{eq:Kl2Z} &&\|w(t,\cdot)\|_{\infty} \lesssim \<t\>^{1-\frac{n-1}2}\,\left(\|w(t,\cdot)\|_{\Gamma,s,2}+\|Dw(t,\cdot)\|_{\Gamma,s,2}\right) \quad
\text{if~$s+1>n/2$,}
\\
\label{eq:Kl3Z}
&&\|w(t,\cdot)\|_{q} \lesssim \<t\>^{-(n-1)\left(\frac12-\frac1q\right)}\|w(t,\cdot)\|_{\Gamma,s,2}\quad \text{if~} \quad 2\le q<\infty\,, \quad \frac{1}{q}\ge \frac{1}{2}-\frac{s}{n}\ge 0,
\end{eqnarray}
for any $t>0$ and any $w(t,\cdot)$ such that right-hand sides are well-defined. The previous statements can be found in \cite{Z}.

Energy estimates in these spaces are given by
\begin{equation}\label{eq:energy}
\|Du\|_{\Gamma,s,2} \lesssim \|\nabla u_0\|_{\Gamma,s,2}+\|u_1\|_{\Gamma,s,2} + \int_0^t \|f(\tau,x)\|_{\Gamma,s,2}\,d\tau,\qquad s\in\N
\end{equation}
for solutions to the Cauchy problem for the inhomogeneous wave equation
\begin{equation}\label{eq:CPinhwave}
\begin{cases}
u_{tt}-\triangle u = f(t,x), & t\geq0, \quad x\in\R^n,\\
u(0,x)=u_0(x), & x\in\R^n,\\
u_t(0,x)=u_1(x), & x\in\R^n.
\end{cases}
\end{equation}
Indeed, we may combine \eqref{eq:comLG}, \eqref{eq:comDG} with the classical energy estimate
\[ \|Du\|_2 \lesssim \|\nabla u_0\|_2+\|u_1\|_{L^2} + \int_0^t \|f(\tau,x)\|_{L^2}\,d\tau.\]
It is also necessary to estimate $\|u\|_{\Gamma,s,2}$. Here the space dimension $n=2$ comes into play.
\begin{lemma}\label{lem:CPlin}
Let~$n=2$, and~$u$ be the solution to~\eqref{eq:CPinhwave}. Then, for any~$\epsilon>0$ there exists~$\delta(\epsilon)>0$, satisfying~$\delta(\epsilon)\to0$ as~$\epsilon\to0$, such that
\[ \|u(t,\cdot)\|_{\Gamma,s,2} \lesssim \|u_0\|_{\Gamma,s,2} + t^\delta \|u_1\|_{\Gamma,s,(1+\epsilon,2)} + \int_0^t (t-\tau)^\delta\,\|f(\tau,\cdot)\|_{\Gamma,s,(1+\epsilon,2)}\,d\tau\,.\]
\end{lemma}
\begin{proof}
Due to \eqref{eq:comLG} it suffices to consider the case $s=0$. First, let~$f\equiv0$. Following~\cite{LX}, by using the change of variables $x=ty$, we may
estimate
\[ \|u(t,\cdot)\|_{2} \lesssim \|u_0\|_{2} + t^2 \|G\|_{H^{-1}(\R^2_y)}\,, \qquad \text{where~$G(y)=u_1(ty)$.} \]
Recalling that
\[ \|G\|_{H^{-1}} = \sup_{v\in H^1,v\neq0} \frac{\int_{\R^2}G(y)v(y)\,dy}{\|v\|_{H^1}}, \]
by virtue of \eqref{eq:Holder}-\eqref{eq:emb} and Sobolev embeddings it holds
\[ \|Gv\|_{L^1} \lesssim \|G\|_{(q,2)}\,\|v\|_{(q',2)} \lesssim \|G\|_{(q,2)}\,\|v\|_{q'} \lesssim \|G\|_{(q,2)}\,\|v\|_{H^1}\,, \]
where~$q=1+\epsilon$ with some~$\epsilon\in(0,1)$. Since
\[ \|G\|_{(q,2)} \lesssim t^{-\frac{n}q}\,\|u_1\|_{(q,2)}\,,\]
summarizing, we proved that
\[ \|u(t,\cdot)\|_{2} \lesssim \|u_0\|_{2} + t^{2(1-\frac1{1+\epsilon})} \|u_1\|_{(1+\epsilon,2)}. \]
The case $f\not\equiv 0$ follows by Duhamel's principle.
\end{proof}
Now we come back to the semi-linear problem and for any~$T>0$, we introduce the space~$X(T)$ with norm
\[ \|u\|_{X(T)} \doteq \sup_{t\in[0,T]} \left( \<t\>^{-\delta}\|u\|_{\Gamma,1,2}+\|Du\|_{\Gamma,1,2}\right), \]
where~$\delta$ is given by Lemma~\ref{lem:CPlin}. For any~$w\in X(T)$ let~$u=S[w]$ be the solution to
\[ u_{tt}-\triangle u = \<t\>^{-(p-1)}|w|^p, \qquad u(0,x)=u_0(x),\quad u_t(0,x)=u_1(x) \]
with compactly supported data. Thanks to Lemma~\ref{lem:CPlin} for~$s=1$ we may estimate
\[ \|u(t,\cdot)\|_{\Gamma,1,2} \lesssim \|u_0\|_{\Gamma,1,2} + t^\delta \|u_1\|_{\Gamma,1,(1+\epsilon,2)} + \int_0^t (t-\tau)^\delta\,\|\<\tau\>^{-(p-1)}|w(\tau,\cdot)|^p\|_{\Gamma,1,(1+\epsilon,2)}\,d\tau. \]
Since
\[ [\p_\tau,\<\tau\>^\alpha]=\alpha\frac1{\<\tau\>}\<\tau\>^\alpha, \qquad [L_0,\<\tau\>^\alpha] = \alpha \<\tau\>^\alpha, \qquad [L_j,\<\tau\>^\alpha]=\alpha\frac{x_j}{\<\tau\>}\,\<\tau\>^\alpha\,, \]
thanks to the finite speed of propagation, i.e. $|x|\lesssim \<t\>$ in~$\supp u$, we get
\[ \|u(t,\cdot)\|_{\Gamma,1,2} \lesssim \|u_0\|_{\Gamma,1,2} + t^\delta \|u_1\|_{\Gamma,1,(1+\epsilon,2)} + \int_0^t (t-\tau)^\delta\,\<\tau\>^{-(p-1)}\,\||w(\tau,\cdot)|^p\|_{\Gamma,1,(1+\epsilon,2)}\,d\tau. \]
Now we may estimate
\[ \||w(\tau,\cdot)|^p\|_{\Gamma,1,(1+\epsilon,2)} \lesssim \||w|^{p-1}\|_{(\bar{q},\infty)} \, \|w(\tau,\cdot)\|_{\Gamma,1,2}, \]
where~$\bar q(\epsilon)\in(2,\infty)$ is given by
\[ \frac1{1+\epsilon}=\frac12+\frac1{\bar q}. \]
Let~$\gamma(\epsilon)\doteq 2/\bar{q}=(1-\epsilon)/(1+\epsilon)$. Since~$p>2$ it holds $\gamma+1<p$. Then
\[ \||w(\tau,\cdot)|^{p-1}\|_{(\bar{q},\infty)} \lesssim \|w(\tau,\cdot)\|_\infty^{p-1-\gamma}\,\|w(\tau,\cdot)\|_{(2,\infty)}^\gamma. \]
Applying Sobolev embeddings on the unit sphere~$S^1$ leads to
\[ \|w(\tau,\cdot)\|_{(2,\infty)}\lesssim \|w(\tau,\cdot)\|_{H^1} \leq \|w(\tau, \cdot)\|_2+\|Dw(\tau,\cdot)\|_2.\]
Thanks to \eqref{eq:Kl2Z} we have
\[ \|w(\tau,\cdot)\|_\infty \lesssim \<\tau\>^{\frac12}\left(\|w(\tau, \cdot)\|_{\Gamma,1,2}+\|Dw(\tau,\cdot)\|_{\Gamma,1,2}\right), \]
therefore, taking into account that~$w\in X(T)$ we conclude
\begin{align*}
\int_0^t (t-\tau)^\delta\,\<\tau\>^{-(p-1)}\,\||w(\tau,\cdot)|^p\|_{\Gamma,1,(1+\epsilon,2)}\,d\tau
    & \lesssim \|w\|_{X(T)}^p\,\int_0^t (t-\tau)^\delta\,\<\tau\>^{-(p-1)+p\delta+\frac{p-(1+\gamma)}2}\,\,d\tau.
\end{align*}
Since~$\delta(\epsilon)\to0$ and~$\gamma(\epsilon)\to1$ as~$\epsilon\to0$, for any~$p>2$, one can find a sufficiently small~$\epsilon$ such that
\[ -(p-1)+p\delta+\frac{p-(1+\gamma)}2< -1\,. \]
To estimate~$\|Du\|_{\Gamma,1,2}$ we apply~\eqref{eq:energy}. Now
\[ \||w(\tau,\cdot)|^p\|_{\Gamma,1,2} \lesssim \||w(\tau,\cdot)|^{p-1}\|_{2+\epsilon_1}\,\|w(\tau,\cdot)\|_{\Gamma,1,q_1} \]
for some~$\epsilon_1>0$, where~$q_1(\epsilon_1)$ is such that
\[ \frac12 = \frac1{2+\epsilon_1}+\frac1{q_1}. \]
Sobolev embeddings yield
\[ \|w(\tau,\cdot)\|_{\Gamma,1,q_1} \lesssim \|w(\tau,\cdot)\|_{\Gamma,1,2}+\|Dw(\tau,\cdot)\|_{\Gamma,1,2}\,.\]
On the other hand, since $p>2$, we have
\[ \||w(\tau,\cdot)|^{p-1}\|_{2+\epsilon_1} \leq \|w(\tau,\cdot)\|_{(2+\epsilon_1)(p-1)}^{p-1} \lesssim \<\tau\>^{-\left(\frac12-\frac1{(2+\epsilon_1)(p-1)}\right)(p-1)}\|w(\tau,\cdot)\|_{\Gamma,1,2}^{p-1} \leq \|w(\tau,\cdot)\|_{\Gamma,1,2}^{p-1}. \]
In turn, this gives
\[ \int_0^t \<\tau\>^{-(p-1)}\,\||w(\tau,\cdot)|^p\|_{\Gamma,1,2}\,d\tau \lesssim \|w\|_{X(T)}^p\,\int_0^t \<\tau\>^{-(p-1)+p\delta}\,\,d\tau.\]
Since~$p>2$, it is sufficient to fix~$\epsilon$ such that~$\delta(\epsilon)$ satisfies $p(1-\delta)>2$.
\\
Summarizing, we proved that
\begin{align*}
\|u(t,\cdot)\|_{\Gamma,1,2}
    & \lesssim \|u_0\|_{\Gamma,1,2} + \<t\>^\delta \|u_1\|_{\Gamma,1,(1+\epsilon,2)} + \<t\>^\delta \|w\|_{X(T)}^p\,,\\
\|Du(t,\cdot)\|_{\Gamma,1,2}
    & \lesssim \|\nabla u_0\|_{\Gamma,1,2} + \|u_1\|_{\Gamma,1,2} + \|w\|_{X(T)}^p\,.
\end{align*}
Recalling that initial data are compactly supported, we derive
\[ \|u\|_{X(T)} \lesssim \left(\|u_0\|_{\Gamma,1,2}+\|Du_0\|_{\Gamma,1,2} + \|u_1\|_{\Gamma,1,(1+\epsilon,2)}+ \|u_1\|_{\Gamma,1,2} \right) + \|w\|_{X(T)}^p\,. \]
By a standard argument, this estimate guarantees that with small data the operator $S[w]$ has a unique fixed point, that is the required solution.

\section{Proof of Theorem~\ref{Thm:3}}
\begin{remark} \label{Rem1}
In the statement of Theorem~\ref{Thm:3} we may relax the assumptions of compact support of the initial data. More precisely, we will prove that for
any~$p>p_0(5)$ and for any~$\kappa\geq (3-p)/(p-1)$ if~$p<2$, or~$\kappa>1$ if~$p\geq2$, there exists~$\varepsilon_0>0$ such that
if~$(v_0,v_1)\in\mathcal{C}^2(\R^3)\times\mathcal{C}^1(\R^3)$ are radial, namely, $v_0=v_0(|x|)$, $v_1=v_1(|x|)$, and
\begin{equation}\label{eq:datau0u1}
\<r\>^{\kappa+1}\left(|v_0(r)+v_1(r)|+\<r\>|v^\prime_0(r)+v^\prime_1(r)| \right) +
\<r\>^\kappa\left(|v_0(r)|+\<r\>|v^\prime_0(r)|+\<r\>^2|v^{\prime\prime}_0(r)|\right)<\varepsilon,
\end{equation}
for some $\varepsilon<\varepsilon_0$, then~\eqref{eq:CP2} admits a (radial) global,
solution~$v\in\mathcal{C}([0,\infty)\times\R^3)\cap\mathcal{C}^2([0,\infty)\times(\R^3\setminus\{0\}))$. We set~$r=|x|$ in~\eqref{eq:datau0u1}.
\end{remark}
\begin{remark} \label{Rem2}
We may also replace the nonlinear term~$|u|^p$ in~\eqref{eq:CP2} by~$f(u)$, where~$f\in\mathcal{C}^1$ is an even function satisfying~$|f^{(h)}(u)|\lesssim
|u|^{p-h}$ for $h=0,1$. In particular, it holds
\begin{equation}\label{eq:flip}
f(0)=0, \qquad |f(u)-f(v)|\lesssim |u-v|(|u|^{p-1}+|v|^{p-1}).
\end{equation}
\end{remark}
To fulfill our objective we apply to~\eqref{eq:CPnew} a technique introduced by Asakura~\cite{A} and developed in different works, in particular
in~\cite{K5}. For the sake of simplicity, let~$v_0=0$ and let~$g\doteq u_1=v_1$. Then condition~\eqref{eq:datau0u1} becomes
\begin{equation}\label{eq:greg}
|g^{(h)}(r)|\leq \varepsilon \<r\>^{-(\kappa+1+h)}\,\,\,\mbox{for}\,\,\, h=0,1.
\end{equation}
We extend~$g$ to negative values of~$r$ by defining $g(r) \doteq g(-r)$ for any~$r<0$. Then, by symmetry, we rewrite~\eqref{eq:CPnew} as
\begin{equation}\label{eq:CPr}
\begin{cases}
u_{tt}-u_{rr}-\frac2r\,u_r = \<t\>^{-(p-1)}|u|^p, & t\geq0, \quad r\in\R,\\
u(0,r)=0, \quad u_t(0,r)=g(r),& r\in\R.
\end{cases}
\end{equation}
\begin{definition}\label{def:solnonl}
We say that $u(t,|x|)=u(t,r)$ is a radial global solution to~\eqref{eq:CPr} if $u\in\mathcal{C}([0,\infty)\times\R)$,
$r^2u\in\mathcal{C}^2([0,\infty)\times\R)$ and
\begin{equation*}
\begin{cases}
r^2u_{tt}-(r^2u_{rr}+2r\,u_r) =r^2\<t\>^{-(p-1)}|u|^p, & t\geq0, \quad r\in\R, \\
u(0,r)=0, \quad u_t(0,r)=g(r), & r\in\R.
\end{cases}
\end{equation*}
\end{definition}
\begin{remark} \label{Rem3}
Any solution to~\eqref{eq:CPr} in the sense of Definition~\ref{def:solnonl} gives
a~$\mathcal{C}([0,\infty)\times\R^3)\cap\mathcal{C}^2([0,\infty)\times(\R^3\setminus\{0\}))$ solution to~\eqref{eq:CPnew} and in turn of to~\eqref{eq:CP2}.
\end{remark}
\subsection{The linear equation}
\begin{definition}\label{def:solin}
Let us consider
\begin{equation}\label{eq:CPrlin}
\begin{cases}
u_{tt}-u_{rr}-\frac2r\,u_r = 0, & t\geq0, \quad r\in\R,\\
u(0,r)=0, \quad u_t(0,r)=g(r), & r\in\R.
\end{cases}
\end{equation}
We say that~$u\in\mathcal{C}([0,\infty)\times\R)$ is a solution to~\eqref{eq:CPrlin} if~$r^2u\in\mathcal{C}^2([0,\infty)\times\R)$ and
\begin{equation}\label{eq:CPr2lin}
\begin{cases}
r^2u_{tt}-(r^2u_{rr}+2r\,u_r) =0, & t\geq0, \quad r\in\R, \\
u(0,r)=0, \quad u_t(0,r)=g(r), & r\in\R.
\end{cases}
\end{equation}
\end{definition}

We see that $r^2u\in\mathcal{C}^2$ and~$u\in\mathcal{C}$ give sufficient regularity for solutions to the equation in \eqref{eq:CPr2lin}. Indeed, we have
$ru_r=\partial_r(r^2u)-2ru \in \mathcal C$ and $r^2 u_{rr}+2ru_r=\partial_{rr}(r^2u)-2u-2ru_r\in \mathcal C$. Hence, also $r^2u_{tt}\in \mathcal{C}$.
According to Definition~\ref{def:solin} the function
\begin{equation*}
u^\lin(t,r)= \int_{-1}^1 H_g(t+r\sigma)\,d\sigma\,\,\, \text{with } \,\,\,H_g(\rho)\doteq \frac{\rho g(\rho)}2
\end{equation*}
is the solution to~\eqref{eq:CPrlin}. This result can be found in \cite{CH}, but we rewrite the
computation for completeness.
Indeed, for any~$H=H(\rho)$, $H\in\mathcal{C}^1$, we put
\begin{equation}\label{eq:v}
v(t,r)\doteq \frac1r\int_{t-r}^{t+r}H(\rho)\,d\rho  =\int_{-1}^1 H(t+r\sigma)\,d\sigma\,.
\end{equation}
For any~$r\neq0$ it holds
\begin{align}
\label{eq:vt}
v_t
    & = \int_{-1}^1 H'(t+r\sigma)\,d\sigma = \frac1r\, (H(t+r)-H(t-r))\,,\\
\label{eq:vtt}
v_{tt}
    & = \frac1r\, (H'(t+r)-H'(t-r))\,,\\
\label{eq:vr}
v_r
    & = \int_{-1}^1 \sigma\,H'(t+r\sigma)\,d\sigma = \frac1r\, (H(t+r)+H(t-r)) - \frac1r\,v\,, \\
\label{eq:vrr} & \begin{array} {cr} v_{rr}
     = -\frac1{r^2}\, (H(t+r)+H(t-r))+\frac1r\, (H'(t+r)-H'(t-r)) + \frac1{r^2}\,v-\frac1r\,v_r \\
     = \frac1r\, (H'(t+r)-H'(t-r)) -\frac2r\,v_r.\end{array}
\end{align}
In particular, $v$ solves the equation in~\eqref{eq:CPr2lin} for any~$r\neq0$. Moreover, $r^2v\in\mathcal{C}^2([0,\infty) \times \R)$, and~$v$ solves the equation
in~\eqref{eq:CPr2lin} for any~$r\in\R$, as one may immediately check by multiplying \eqref{eq:vt}-\eqref{eq:vr} by~$r$ and \eqref{eq:vtt}-\eqref{eq:vrr}
by~$r^2$. We remark that $v(0,r)=0$ if~$H$ is odd. In this latter case, $rv_t(0,r)=2H(r)$.  In particular, this proves that~$u^\lin$
solves~\eqref{eq:CPrlin}.

For our convenience we also compute
\begin{align}
\label{eq:derv}
\partial_r(rv)
    & = v + rv_r = H(t+r)+H(t-r),\\
\label{eq:de2r2v}
\p_r^2(r^2v)
    & = \p_r(rv)+r\p_r^2(rv) = H(t+r)+H(t-r)+r(H'(t+r)-H'(t-r)).
\end{align}
For any fixed $\kappa>1$, we introduce the Banach space
\begin{equation*}
 X_\kappa \doteq \left\{ u\in C([0,\infty),\R), \ \, u \,\,\text{is even in~$r$} : \ \partial_r(ru)\in\mathcal{C}([0,\infty),\R), \; \|u\|_{X_\kappa}<\infty \right\}
\end{equation*}
with the norm
\begin{equation*}
\|u\|_{X_\kappa}\doteq \n{\<t+|r|\>\,\<t-|r|\>^{\kappa-1} u} + \n{\<r\>^{-1}\<t+|r|\>\,\<t-|r|\>^{\kappa-1} \partial_r(ru)}.
\end{equation*}
\begin{theorem}\label{thm:lin}
Suppose that~\eqref{eq:greg} holds for some $\kappa>1$. Then
\begin{equation*}
\|u^\lin\|_{X_\kappa}\leq C\varepsilon
\end{equation*}
for a suitable constant $C>0$.
\end{theorem}
\begin{proof}
We notice that
\begin{equation*}
|H_g^{(h)}(\rho)|\leq \varepsilon \<\rho\>^{-\kappa-h} \,\,\,\mbox{for}\,\,\, h=0,1.
\end{equation*}
Thanks to \eqref{eq:derv} we immediately derive
\begin{equation*}
|\partial_r(ru^\lin)| = |H_g(t+r)+H_g(t-r)| \lesssim \,\varepsilon\,\<t-|r|\>^{-\kappa}.
\end{equation*}
We distinguish two cases. \\
If~$t\geq 2|r|$, then~$\<t\pm|r|\>\simeq \<t\>$ and we get
\begin{equation*}
|\partial_r(ru^\lin)| \lesssim \,\varepsilon\,\<t-|r|\>^{-(\kappa-1)}\<t+|r|\>^{-1}\lesssim \,\varepsilon\,\<t-|r|\>^{-(\kappa-1)}\<t+|r|\>^{-1}\,\<r\>\,,
\end{equation*}
where in the last inequality we used the trivial estimate~$1\leq\<r\>$. \\
If~$t\leq 2|r|$, then~$\<t+|r|\>\leq 3\<r\>$, therefore,
\begin{equation*}
|\partial_r(ru^\lin)| \lesssim \,\varepsilon\,\<t-|r|\>^{-\kappa}\,\<t+|r|\>^{-1}\,\<r\> \lesssim
\,\varepsilon\,\<t-|r|\>^{-(\kappa-1)}\,\<t+|r|\>^{-1}\,\<r\>,
\end{equation*}
where in the last inequality we use the trivial estimate~$\<t-|r|\>^{-1}\leq1$.

In order to estimate $\n{\<t+|r|\>\<t-|r|\>^{\kappa-1}u^\lin}$ we observe that
\begin{equation*}
|u^\lin(t,r)|\lesssim \frac{\varepsilon}r \int_{t-r}^{t+r}\<\rho\>^{-\kappa}\,d\rho
= \frac1{|r|}\,C\,\varepsilon\, \int_{t-|r|}^{t+|r|}\<\rho\>^{-\kappa}\,d\rho.
\end{equation*}
If~$t\geq 2|r|$, then~$\<t\pm|r|\>\simeq \<t\>$, hence,
\begin{equation*}
|u^\lin(t,r)| \simeq \varepsilon \<t-|r|\>^{-(\kappa-1)}\,\<t+|r|\>^{-1}.
\end{equation*}
If~$t\leq 2|r|$, then we also distinguish two cases. If~$|r|\leq1$, then~$\<t+|r|\>\<t-|r|\>^{\kappa-1}\simeq 1$ and it is sufficient to estimate
\begin{equation*}
|u^\lin(t,r)| \leq \int_{-1}^1 H_g(t+r\sigma)\,d\sigma \leq C.
\end{equation*}
On the other hand, if~$t\leq2|r|$ and~$|r|\geq1$, then~$\<t+|r|\>\leq3\<r\>$ and~$|r|\simeq \<r\>$, therefore,
\begin{align*}
|u^\lin(t,r)|
    & \lesssim \frac1{|r|}\,\varepsilon\, \int_{t-|r|}^{t+|r|}\<\rho\>^{-\kappa}\,d\rho \simeq \frac1{\<r\>}\,\varepsilon\,\<t-|r|\>^{-(\kappa-1)}\\
    & \lesssim \varepsilon\,\<t+|r|\>^{-1}\,\<t-|r|\>^{-(\kappa-1)},
\end{align*}
thanks to~$\kappa>1$.
This concludes the proof that~$\|u^\lin\|_{X_\kappa}\leq C\varepsilon$.
\end{proof}

\subsection{Duhamel's principle and basic nonlinear estimates}
For any~$u\in X_\kappa$ let
\[
Lu(t,r) \doteq \int_0^t \<s\>^{-(p-1)}\,\int_{-1}^1 H_u[s](t-s+r\sigma)\,d\sigma\,ds=\frac{1}{r}\int_0^t \<s\>^{-(p-1)}\int_{t-s-r}^{t-s+r}H_u[s](\rho)\,d\rho \,ds,
\]
where
\begin{equation}\label{eq:Hf}
H_u[s](\rho) \doteq \frac{\rho\,f(u(s,\rho))}2\,.
\end{equation}
We denote by $H_u[s]'(\rho)$ the derivative of $H_u[s](\rho)$ with respect to $\rho$, considering $s$ as a parameter.

Let us consider $f(u(s,\rho))$ and $\rho\partial_\rho f(u(s,\rho))$. If $u\in X_\kappa$, recalling that $ru_r=\p_r(ru)-u$, then we may estimate
\begin{align*}
|f(u(s,\rho))|
    & \lesssim \|u\|^p_{X_\kappa}\,\<s+|\rho|\>^{-p}\,\<s-|\rho|\>^{-p(\kappa-1)},\\
\<\rho\>^{-1}\,|\rho\,\p_\rho f(u(s,\rho))|
    & \lesssim \|u\|^p_{X_\kappa}\,\<s+|\rho|\>^{-p}\,\<s-|\rho|\>^{-p(\kappa-1)},
\end{align*}
Having in mind~\eqref{eq:Hf}, it follows, in particular, that
\begin{equation}\label{eq:basicH}
|H_u[s](\rho)| + |H_u[s]'(\rho)| \lesssim \|u\|^p_{X_\kappa}\,\<s+|\rho|\>^{-p}\,\<s-|\rho|\>^{-p(\kappa-1)}\<\rho\>.
\end{equation}
\begin{prop}\label{prop:Duh}
Let $u\in X_\kappa$ be even with respect to $r$. Then $Lu\in X_\kappa$ and $r^2Lu \in\mathcal{C}^2([0,\infty)\times\R)$. Moreover, $Lu$ is even with
respect to $r$ and satisfies
\begin{equation}\label{eq:CPr2}
r^2\left(\p_t^2-\p_r^2\right)Lu-2r \partial_r Lu = \<t\>^{-(p-1)}\,r^2f(u),\qquad t\geq0, \quad r\in\R
\end{equation}
with zero initial data $g=0$.
\end{prop}
\begin{proof}
From the continuity of $H_u[s](\rho)$ (which follows from~$u\in X_\kappa \subset \mathcal{C}$), it follows that~$Lu \in X_\kappa$, i.e. $Lu, \p_r(rLu)\in
\mathcal{C}$. Being~$u$ even with respect to~$r$, and $f$ even in $u$, we get that $H_u[s]$ is odd for any $s$. It follows that $Lu$ is even. We notice
that
\begin{align*}
\p_t Lu
    & = \int_0^t \<s\>^{-(p-1)}\,\int_{-1}^1 \p_t H_u[s](t-s+r\sigma)\,d\sigma\,ds + \<t\>^{-(p-1)}\,\int_{-1}^1 H_u[t](r\sigma)\,d\sigma \\
    & = \frac1r\,\int_0^t \<s\>^{-(p-1)}\,\bigl(H_u[s](t-s+r)-H_u[s](t-s-r)\bigr)\,ds, \\
\p_t^2 Lu
    & = \frac1r\,\int_0^t \<s\>^{-(p-1)}\,\bigl(H_u[s]'(t-s+r)-H_u[s]'(t-s-r)\bigr)\,ds + \frac1r \<t\>^{-(p-1)}\,\bigl(H_u[t](r)-H_u[t](-r)\bigr) \\
    & = \frac1r\,\int_0^t \<s\>^{-(p-1)}\,\bigl(H_u[s]'(t-s+r)-H_u[s]'(t-s-r)\bigr)\,ds + \<t\>^{-(p-1)}\,f(u(t,r)).
\end{align*}
In particular, we gain $\p_t^2 Lu\in \mathcal C$. Recalling \eqref{eq:vrr}, we see that~$Lu$ solves~\eqref{eq:CPr2} and we get the continuity of the
$r$-derivatives for $r^2Lu$.
\end{proof}

In order to prove global (in time) existence trough contraction mapping principle we shall prove the following statement.
\begin{theorem}\label{thm:nonl}
Let $p>p_0(5)$ and let
\begin{equation}\label{eq:rangek}
\frac{3-p}{p-1} \leq \kappa \leq 2(p-1) \quad \text{if~$p\in(p_0(5),2)$, or~$1<\kappa\leq2(p-1)$ if~$p\geq2$.}
\end{equation}
If $u\in X_\kappa$, then
\begin{align}
\label{eq:basicLu}
&\|Lu\|_{X_\kappa}\lesssim \|u\|_{X_\kappa}^p;\\
\label{eq:basicLuLv}
&\|Lu-Lv\|_{X_\kappa} \lesssim \|u-v\|_{X_\kappa}\left(\|u\|_{X_\kappa}^{p-1}+\|u\|_{X_\kappa}^{p-1}\right).
\end{align}
\end{theorem}
Recalling the definition of the involved norm, for proving \eqref{eq:basicLu} it suffices to show
\begin{align}
&|Lu(t,r)|\lesssim \<t+|r|\>^{-1}\<t-|r|\>^{-(\kappa-1)}\|u\|_{X_\kappa}^{p}\,,
\label{eq:stimaL}\\
&|\partial_r(rLu)(t,r)|\lesssim \<t+|r|\>^{-1}\<t-|r|\>^{-(\kappa-1)}\<r\>\|u\|_{X_\kappa}^{p}\,.
\label{eq:stimaL1}
\end{align}
Since $Lu$ is even in $r$, it suffices to deal with $r>0$. Proceeding as in \eqref{eq:v}, from \eqref{eq:basicH} we have
\begin{align*}
|Lu(t,r)| &\lesssim  \frac{1}{r}\int_0^t \<s\>^{-(p-1)}\,\int_{t-s-r}^{t-s+r}| H_u[s](\rho)|\,d\rho\,ds\\
& \lesssim  \frac{1}{r}\|u\|^p_{X_\kappa}
\int_0^t \<s\>^{-(p-1)}\,\int_{t-s-r}^{t-s+r} \<s+|\rho|\>^{-p}\, \<s-|\rho|\>^{-p(\kappa-1)} \<\rho\> \,d\rho\,ds.
\end{align*}
By using \eqref{eq:derv} we get
\begin{align*}
|\p_r(rLu)(t,r)|& \le  \int_0^t \<s\>^{-(p-1)}\,\bigl|H_u[s](t-s+r)+H_u[s](t-s-r)|\,ds \\
& \lesssim \|u\|^p_{X_\kappa} \sum_{\pm}\int_0^t \<s\>^{-(p-1)}\,\<s+|t-s\pm r|\>^{-p}\<s-|t-s\pm r|\>^{-p(\kappa-1)}\<t-s\pm r\>ds.
\end{align*}
Consequently, our aim reduces to estimate the quantities
\begin{align*}
 I_0(t,r) &\doteq \int_0^t \<s\>^{-(p-1)}\,\int_{t-s-r}^{t-s+r} \<s+|\rho|\>^{-p}\, \<s-|\rho|\>^{-p(\kappa-1)} \<\rho\> \,d\rho\,ds,\\
 I_{1,\pm}(t,r) &\doteq \int_0^t \<s\>^{-(p-1)}\,\<s+|t-s\pm r|\>^{-p}\<s-|t-s\pm r|\>^{-p(\kappa-1)}\<t-s\pm r\>ds.
\end{align*}

Similarly, to prove \eqref{eq:basicLuLv} it suffices to show
\begin{align}
&|Lu(t,r)-Lv(t,r)|\lesssim \<t+|r|\>^{-1}\<t-|r|\>^{-(\kappa-1)}\|u-v\|_{X_\kappa}\left(\|u\|_{X_\kappa}^{p-1}+\|v\|_{X_\kappa}^{p-1}\right)\,,
\label{eq:stimaLuLv}\\
&|\partial_r(rLu)(t,r)- \partial_r(rLv)(t,r)|\lesssim \<t+|r|\>^{-1}\<t-|r|\>^{-(\kappa-1)}\<r\>\|u-v\|_{X_\kappa}\left(\|u\|_{X_\kappa}^{p-1}+\|v\|_{X_\kappa}^{p-1}\right)\,.
\label{eq:stima1LuLv}
\end{align}
We have
\begin{equation*}
|Lu(t,r)-Lv(t,r)| \lesssim  \frac{1}{r}\int_0^t \<s\>^{-(p-1)}\,\int_{t-s-r}^{t-s+r}| H_u[s](\rho)-H_v[s](\rho)|\,d\rho\,ds.
\end{equation*}
Moreover, from \eqref{eq:flip} and \eqref{eq:Hf}, it follows that
\begin{equation*}
| H_u[s](\rho)-H_v[s](\rho)|\lesssim |\rho||u(s,\rho)-v(s,\rho)|\left(|u(s,\rho)|^{p-1}+|v(s,\rho)|^{p-1}\right)\,.
\end{equation*}
As a conclusion
\begin{equation*}
|Lu(t,r)-Lv(t,r)| \lesssim \|u-v\|_{X_\kappa}\left(\|u\|_{X_\kappa}^{p-1}+\|v\|_{X_\kappa}^{p-1}\right) I_0(t,r).
\end{equation*}
Similarly, we get
\begin{equation*}
|\partial_r(rLu(t,r)-rLv(t,r))| \lesssim \|u-v\|_{X_\kappa}\left(\|u\|_{X_\kappa}^{p-1}+\|v\|_{X_\kappa}^{p-1}\right) \sum_{\pm} I_{1,\pm}(t,r).
\end{equation*}
If~$t\leq r$, then we may simplify our approach, thanks to the following.
\begin{remark}\label{rem:tr}
If~$t\leq r$, it holds
\[ Lu(t,r) =\frac{1}{r}\int_0^t \<s\>^{-(p-1)}\int_{r-(t-s)}^{r+(t-s)}H_u[s](\rho)\,d\rho \,ds. \]
Indeed,
\[ \int_{(t-s)-r}^{r-(t-s)} H_u[s](\rho)\,d\rho =0, \]
being~$H_u[s]$, defined in~\eqref{eq:Hf} odd, thanks to the assumption that~$f(u)$ is even with respect to~$u$, and thanks to the fact that~$u$ is even with respect to~$r$. Therefore, we may replace~$I_0(t,r)$ by
\[ I_0'(t,r) \doteq \int_0^t \<s\>^{-(p-1)}\,\int_{r-(t-s)}^{r+(t-s)} \<s+|\rho|\>^{-p}\, \<s-|\rho|\>^{-p(\kappa-1)} \<\rho\> \,d\rho\,ds. \]
\end{remark}
The estimates for $I_0$, $I_{1,\pm}$ and~$I_0'$ are based on the following lemma.
\begin{lemma}\label{lem:Ixi}
Let $p>p_0(5)$ and let
\begin{gather}
\label{eq:rangek1}
\frac{3-p}{p-1} \le \kappa \leq 2(p-1) \qquad \text{if~$p\in(p_0(5),2)$,}\\
\label{eq:rangek2}
1<\kappa\leq 2 \qquad \text{if~$p=2$,}\\
\label{eq:rangek3}
\frac1{p-1} \le \kappa \leq 2(p-1) \qquad \text{if~$p>2$.}
\end{gather}
Then
\begin{equation}\label{eq:alpha+1}
I(\xi)=\int_{-\xi}^{\xi} \<\eta+\xi\>\<\eta-\xi\>^{-(p-1)}\<\eta\>^{-p(\kappa-1)}\,d\eta\lesssim \<\xi\>^{-(\kappa-p)}.
\end{equation}
\end{lemma}
\begin{remark}
If~$p<2$, then the interval~\eqref{eq:rangek1}, i.e.~$(3-p)(p-1)^{-1}\leq \kappa\leq 2(p-1)$, is nonempty if, and only if, $p>p_0(5)$. If~$p>2$, then the
interval~\eqref{eq:rangek3}, i.e.~$(p-1)^{-1}\leq \kappa\leq 2(p-1)$ is nonempty for any~$p>2$. \\ We observe that this latter range contains the
range~$(1,2(p-1)]$ required in the assumption \eqref{eq:rangek}.
\end{remark}
\begin{proof}
We split $I(\xi)$ in $I_1(\xi)=\int_{-\xi/2}^{\xi/2} \dots \,d\eta$ and $I_2(\xi)$ as the remainder.

Let $\eta \in [0,\xi/2]$. Then we have $\<\xi \>\simeq \<\eta+\xi\>\simeq \<\xi-\eta\>$. Hence,
\begin{equation*}
I_{1}(\xi)\lesssim \<\xi\>^{2-p}\int_{0}^{\xi/2}\<\eta\>^{-p(\kappa-1)}\,d\eta.
\end{equation*}
We get $I_1(\xi)\lesssim \<\xi\>^{-(\kappa-p)}$ if
\begin{align*}
\kappa<1+\frac{1}{p} & \text{ and } -3+p\kappa\ge \kappa-p,\\
\kappa=1+\frac{1}{p} & \text{ and } -2+p>\kappa-p,\\
\kappa>1+\frac{1}{p} & \text{ and } -2+p\geq\kappa-p.
\end{align*}
The first condition corresponds to the interval
\[ \left[\frac{3-p}{p-1},1+\frac1p\right), \]
which is nonempty for any~$p>p_0(5)$. The second condition holds for any~$p>p_0(5)$, therefore~$\kappa=1+1/p$ is admissible. The third condition
corresponds to the interval
\[ \left(1+\frac1p, 2(p-1)\right], \]
which is nonempty for any~$p>p_0(5)$. Gluing together the above three intervals we obtain the admissible range in~\eqref{eq:rangek1}, i.e.
\[ \left[\frac{3-p}{p-1},2(p-1)\right]. \]
Now, let $\eta\in[\xi/2,\xi]$. We have $\<\eta\>\simeq \<\xi\>\simeq\<\xi+\eta\>$. It follows that
\begin{equation*}
I_2(\xi)\simeq\<\xi\>^{1-p(\kappa-1)}\int_{\xi/2}^{\xi} \<\eta-\xi\>^{-(p-1)}\,d\eta+
\<\xi\>^{-(p-1)-p(\kappa-1)}\int_{\xi/2}^{\xi}  \<\eta-\xi\>\,d\eta=I_{2,1}(\xi)+I_{2,2}(\xi).
\end{equation*}
For any $p>1$ we have
\begin{equation*}
I_{2,2}(\xi)\lesssim \<\xi\>^{-(p-1)-p(\kappa-1)+2},
\end{equation*}
in particular, $I_{2,2}\le  \<\xi\>^{-(\kappa-p)}$ for any
\begin{equation}\label{eq:I22k}
\kappa\ge \frac{3-p}{p-1}.
\end{equation}

The estimate of $I_{2,1}$ depends on the range of $p$:
\begin{align*}
I_{2,1}(\xi)\lesssim \<\xi\>^{1-p(\kappa-1)} & \text{ if } p>2,\\
I_{2,1}(\xi)\lesssim \<\xi\>^{1-p(\kappa-1)}\ln\<\xi\> & \text{ if } p=2,\\
I_{2,1}(\xi)\lesssim \<\xi\>^{1-p(\kappa-1)-(p-1)+1}  & \text{ if } p<2.
\end{align*}
For $p<2$ the assumption $\kappa \ge \frac{3-p}{p-1}$ gives directly $I_{2,1}(\xi)\le \<\xi\>^{-(\kappa-p)}$. For~$p=2$, we get~$1-p(\kappa-1)<p-\kappa$ if
and only if $\kappa>1$. For $p>2$, we get~$1-p(\kappa-1)\leq p-\kappa$ if and only if $\kappa\ge \frac{1}{p-1}$.

Therefore, combining the lower bound on~$\kappa$ obtained for~$I_{2,1}$ with the upper bound for~$\kappa$ derived for~$I_1$, we obtain~\eqref{eq:rangek2}
if~$p=2$ and~\eqref{eq:rangek3} if~$p>2$.
\end{proof}

\begin{prop}\label{thm:I0}
Let $p>p_0(5)$ and $\kappa$ be as in~\eqref{eq:rangek}. It holds
\begin{equation*}
I_0(t,r) \lesssim
\begin{cases}
r\,\<t+r\>^{-\kappa} & \text{if $t\geq 2r$\quad \text{ or } $r\le 1$}, \\
\<t-r\>^{-(\kappa-1)}& \text{if $r\leq t\leq 2r$ \text{ and } $r\ge 1$}.
\end{cases}
\end{equation*}
Moreover,
\[ I_0'(t,r) \lesssim \<t-r\>^{-(\kappa-1)}\qquad \text{if $t\leq r$ \text{ and } $r\ge 1$.} \]
In particular, the estimates \eqref{eq:stimaL} and \eqref{eq:stimaLuLv} hold.
\end{prop}
\begin{proof}
First, let us estimate~$I_0$. Being~$|t-s-r|<t-s+r$ we have
\begin{equation*}
I_0(t,r) \le 2 \int_0^t \< s\>^{-(p-1)}\int_{\max\{0,t-s-r\}}^{t-s+r} \<s+\rho\>^{-p}\, \<s-\rho\>^{-p(\kappa-1)} \<\rho\> \,d\rho\,ds.
\end{equation*}
Now we use the change of variables $\xi=s+\rho$, $\eta=\rho-s$. Since $\rho \ge 0$ we have $|\eta|\le \xi$. Moreover, $\xi=s+\rho\le s+(t-s-r)=t+r$ and
$\xi \ge s+\max\{ 0, t-s-r\}\geq (t-r)_+$. Finally, we arrive at
\begin{align}
\label{eq:change} \begin{array}{cr}
I_0(t,r) \lesssim \int_{(t-r)_+}^{t+r}\<\xi\>^{-p}\int_{-\xi}^{\xi} \<\eta+\xi\>\<\eta-\xi\>^{-(p-1)}\<\eta\>^{-p(\kappa-1)}\,d\eta\,d\xi\\
=\int_{(t-r)_+}^{t+r} \<\xi\>^{-p}I(\xi)\,d\xi \end{array}
\end{align}
with $I(\xi)$ as in Lemma~\ref{lem:Ixi}. From Lemma~\ref{lem:Ixi} we conclude
\begin{equation}\label{eq:I0final}
I_0(t,r) \lesssim \int_{(t-r)_+}^{(t+r)} \<\xi\>^{-\kappa}d\xi.
\end{equation}
In the following we shall use different ideas in different zones of the $(t,r)$ plane.

\subsubsection{The zone $t\ge 2r$}
Here, we have in $[(t-r), (t+r)]$ the equivalence $\<\xi\>\simeq \<t+r\>$, therefore,~$I_0(t,r)\lesssim
r\<t+r\>^{-\kappa}$.
\subsubsection{The zone $r\le 1$ and $t\le 2r$} In this zone it is $\<t+r\>\simeq 1$. It is enough to show $I_0(t,r)\lesssim r$, which follows
from~\eqref{eq:I0final} being~$\kappa\geq0$.
\subsubsection{The zone $r\ge 1$ and $r\leq t\le 2r$} If $r\le t\le 2r$, then from~\eqref{eq:I0final} we derive
\begin{equation*}
I_0(t,r) \lesssim \int_{t-r}^{t+r} \<\xi\>^{-\kappa} d\xi \lesssim \<t-r\>^{-(\kappa-1)},
\end{equation*}
where we used~$\kappa>1$.

Now, let us estimate~$I_0'$ for~$r\geq1$ and $t\leq r$. Applying the same change of variables to
\[ I_0'(t,r) = \int_0^t \< s\>^{-(p-1)}\int_{r-(t-s)}^{r+(t-s)} \<s+\rho\>^{-p}\, \<s-\rho\>^{-p(\kappa-1)} \<\rho\> \,d\rho\,ds \]
we obtain
\[ I_0'(t,r) \lesssim \int_{r-t}^{r+t}\<\xi\>^{-p}\int_{r-t}^{\xi} \<\eta+\xi\>\<\eta-\xi\>^{-(p-1)}\<\eta\>^{-p(\kappa-1)}\,d\eta\,d\xi\,.\]
Moreover, $[(r-t), \xi]\subset [-\xi,\xi]$. From Lemma \ref{lem:Ixi} we have
\begin{equation*}
I_0(t,r) \lesssim \int_{r-t}^{r+t}\<\xi\>^{-p}I(\xi)\,d\xi\lesssim \int_{r-t}^{r+t}\<\xi\>^{-\kappa}\,d\xi
\lesssim \<t-r\>^{1-\kappa},
\end{equation*}
where we used again~$\kappa>1$. Finally, we prove \eqref{eq:stimaL}. If $t\ge 2r$ or $r\le 1$, then from $\<t+r\>\ge \<t-r\>$ it follows
\begin{equation*}
|Lu(t,r)|\lesssim \<t+r\>^{-\kappa}\|u\|^p_{X_\kappa}\lesssim \<t+r\>\<t-r\>^{-(\kappa-1)} .
\end{equation*}
For $r\geq1$ and $t\le 2r$ we have
\begin{equation*}
|Lu(t,r)|\lesssim \<r\>^{-1}\,\<t-r\>^{-(\kappa-1)}\|u\|^p_{X_\kappa}\simeq \<t+r\>^{-1}\,\<t-r\>^{-(\kappa-1)}\|u\|^p_{X_\kappa}.
\end{equation*}
The same arguments lead to \eqref{eq:stimaLuLv}.
\end{proof}
\begin{prop}\label{thm:I1}
Let $p>p_0(5)$ and~$\kappa$ be as in~\eqref{eq:rangek}. One has
\begin{equation*}
I_{1,-}(t,r) \lesssim
\begin{cases}
\<t-r\>^{-\kappa} & \text{if $t\geq 2r$,} \\
\<t-r\>^{-(\kappa-1)} & \text{if $t\leq 2r$,}
\end{cases}
\end{equation*}
and~$I_{1,+}\lesssim \<t+r\>^{-\kappa}$. In particular, the estimates \eqref{eq:stimaL1} and~\eqref{eq:stima1LuLv} hold.
\end{prop}

\begin{proof}
We start with the estimate of $I_{1,-}$.\\
\subsubsection{The zone $t\ge 2r$} Since $t+r\simeq t-r$ and if $s\in[t-r,t]$, then
\begin{equation*}
s+|t-s-r|\simeq t-r.
\end{equation*}
Conversely, if $s\in [0,t-r]$, then
\begin{equation*}
s+|t-s-r|=s +t-s-r=t-r.
\end{equation*}
Therefore,
\[ I_{1,-} \lesssim \<t-r\>^{-p} \int_0^t \<s\>^{-(p-1)}\,\<s-|t-s-r|\>^{-p(\kappa-1)}\<t-s-r\>ds=  \<t-r\>^{-p} (Q_-+Q_+),\]
where
\begin{align*}
Q_- & =\int_0^{t-r} \<s\>^{-(p-1)}\,\<2s-t+r\>^{-p(\kappa-1)}\<t-s-r\>ds,\\
Q_+ & =\<t-r\>^{-p(\kappa-1)}\,\int_{t-r}^t \<s\>^{-(p-1)}\,\<t-s-r\>ds.
\end{align*}
We may directly estimate
\begin{align*}
Q_+&\le \<t-r\>^{-p(\kappa-1)-(p-1)}\,\int_{t-r}^t \<t-s-r\>d\tau = \<t-r\>^{-p(\kappa-1)-(p-1)}\,\int_{-r}^0 \<\rho\>d\rho \\
&\lesssim \<t-r\>^{-p(\kappa-1)-(p-1)+2}.
\end{align*}
Being $\kappa\geq\frac{3-p}{p-1}$ we have the required estimate $Q_+\lesssim \<t-r\>^{p-\kappa}$.

In order to estimate $Q_-$ we plan to use Lemma \ref{lem:Ixi}. By the change of variables $\eta=\frac{t-r}{2}-s$, we have
\begin{equation}\label{eq:Q-}
Q_- \lesssim \int_{-\frac{t-r}{2}}^{\frac{t-r}{2}} \Big\<\eta+\frac{t-r}{2}\Big\>^{-(p-1)}\,\<\eta\>^{-p(\kappa-1)}\Big\<\eta-\frac{t-r}{2}\Big\>d\eta=
I\left(\frac{t-r}{2}\right)\lesssim  \<t-r\>^{p-\kappa}.
\end{equation}
Together with the estimate of $Q_+$ this gives $I_{1,-}\lesssim \<t-r\>^{-\kappa}$.
\subsubsection{The zone $t\leq 2r$} We write $I_{1,-}=\tilde Q_++\tilde
Q_-$, where
\begin{align*}
\tilde Q_-
    & =\int_0^{(t-r)_+} \<s\>^{-(p-1)}\,\<t-r\>^{-p}\<2s-t+r\>^{-p(\kappa-1)}\<t-s-r\>ds\\
    & = \<t-r\>^{-p}\,Q_-, \\
\tilde Q_+
    & =\int_{(t-r)_+}^t \<s\>^{-(p-1)}\,\<2s-t+r\>^{-p}\<t-r)\>^{-p(\kappa-1)}\<t-s-r\>ds \\
    & = \<t-r\>^{-p(\kappa-1)}\,\int_{(t-r)_+}^t \<s\>^{-(p-1)}\,\<2s-t+r\>^{-p}\,\<t-s-r\>ds = \<t-r\>^{-p(\kappa-1)}\,Q_+^\sharp.
\end{align*}
Since estimate \eqref{eq:Q-} holds for any $t\ge r$ we may directly conclude $\tilde Q_-\lesssim \<t-r\>^{-\kappa}$.
Since $p>1$, in order to gain $\tilde Q_+\lesssim \<t-r\>^{-(\kappa-1)}$, it suffices to estimate $Q_+^\sharp$ by a constant.\\
Since~$2s-(t-r)\ge s-(t-r)$ we have
\begin{align*}
Q_+^\sharp &\lesssim \int_0^\infty \<s\>^{-(p-1)}\<s-(t-r)\>^{-(p-1)}\,ds\\
& \lesssim \int_0^{(t-r)/2} \<s\>^{-2(p-1)}\,ds+\int_{(t-r)/2}^{\infty} \<t-s-r\>^{-2(p-1)}\,ds\le 2 \int_0^{+\infty} \<s\>^{-2(p-1)}\,ds.
\end{align*}
This quantity is finite taking into consideration $2(p-1) > 2(p_0(5)-1) > 1$.

\medskip
The estimate for $I_{1,+}$ is simpler to obtain. Indeed
\begin{equation*}
I_{1,+} =\<t+r\>^{-p} \int_0^t \<s\>^{-(p-1)}\,\<2s-t-r\>^{-p(\kappa-1)}\<t-s+r\>ds,
\end{equation*}
due to~$t+r-s\geq0$. After the change of variables $\eta=\frac{t+r}{2}-s$ we are in position to apply Lemma \ref{lem:Ixi} and conclude
\begin{equation*}
I_{1,+}  \lesssim \<t+r\>^{-p} \int_{-\frac{t+r}{2}}^{\frac{t+r}{2}}
\Big\<\eta+\frac{t+r}{2}\Big\>^{-(p-1)}\,\<\eta\>^{-p(\kappa-1)}\Big\<\eta-\frac{t+r}{2}\Big\>d\eta= \<t+r\>^{-p} I\left(\frac{t+r}{2}\right) \lesssim
\<t+r\>^{-\kappa}\,.
\end{equation*}

\medskip
Now, we can easily gain \eqref{eq:stimaLuLv}, and similarly \eqref{eq:stima1LuLv}. If $t\ge 2r$, then we use $\<t+r\>\simeq\<t-r\>$ and $\<r\>\ge 1$ to
conclude
\begin{equation*}
|\p_r(rLu)(t,r)|  \lesssim \|u\|^p_{X_\kappa} \sum_{\pm}I_{1,\pm} \lesssim
\|u\|^p_{X_\kappa} \<t-r\>^{-\kappa} \lesssim \<r\> \<t+r\>^{-1}\<t-r\>^{-{\kappa-1}}\|u\|^p_{X_\kappa}\,.
\end{equation*}
If $t\le 2r$, then $\<r\>\simeq \<t+r\>$, hence,
\begin{equation*}
|\p_r(rLu)(t,r)|  \lesssim \|u\|^p_{X_\kappa} \sum_{\pm}I_{1,\pm} \lesssim
\|u\|^p_{X_\kappa} \<t-r\>^{-(\kappa-1)} \lesssim \<r\> \<t+r\>^{-1}\<t-r\>^{-{\kappa-1}}\|u\|^p_{X_\kappa}\,.
\end{equation*}
\end{proof}

\subsection{Existence theorem}

\begin{theorem}\label{thm:final}
Let $p>p_0(5)$ and~$\kappa$ as in~\eqref{eq:rangek}. There exists a constant $\varepsilon_0>0$ such that if~\eqref{eq:greg} holds with
$\varepsilon<\varepsilon_0$, then the Cauchy problem \eqref{eq:CPr} admits a unique global (in time) small data solution $u(t,r)$ in the sense of
Definition \ref{def:solnonl}. In particular, $u\in X_\kappa$ and the following decay estimate holds:
\begin{align}\label{eq:decaynonl}
|u(t,r)|+|\partial_r u(t,r)|&\lesssim \<t+|r|\>^{-1}\,\<t-|r|\>^{-(\kappa-1)}\,.
\end{align}
\end{theorem}

\begin{proof}
Let us define the sequence
\begin{equation*}
u_0=u^\lin,\quad u_{n+1}=u^\lin+Lu_n.
\end{equation*}
By using Theorem \ref{thm:lin} and Theorem \ref{thm:nonl} we get
\begin{align*}
&\|u_{n+1}\|_{X_\kappa}\le \|u^\lin\|_{X_\kappa}+C_1\|u_n\|_{X_\kappa}^p \le C_0\varepsilon+ C_1\|u_n\|_{X_\kappa}^p\,,\\
& \|u_{n+1}-u_n\|_{X_\kappa} \le C_2 \|u_{n}-u_{n-1}\|_{X_\kappa}\left(\|u_{n}\|_{X_\kappa}^{p-1}+\|u_{n-1}\|_{X_\kappa}^{p-1}\right)
\end{align*}
with suitable constant $C_0,C_1,C_2>0$. For $\varepsilon_0<(2C_0C_1^{1/(p-1)})^{-1}$, via induction argument we find
\begin{equation*}
\|u_n\|_{X_\kappa}\le 2\|u^\lin\|_{X_\kappa}\le 2C_0\varepsilon_0.
\end{equation*}
In turn, for $\varepsilon_0<(2^{p+1}C_2 C_0^{p-1})$, this gives
\begin{equation*}
 \|u_{n+1}-u_n\|_{X_\kappa}\le 2^{-n}\|u_1-u_0\|_{X_\kappa}\,.
\end{equation*}
We can conclude that $\{u_n\}$  is a Cauchy sequence, it converges in $X_\kappa$ to the solution to $u=u^\lin+Lu$. According to Proposition \ref{prop:Duh}
this solution is the required one. The decay estimates follow from the definition of $X_\kappa$.
\end{proof}

\begin{remark}\label{rem:final}
From the decay estimate \eqref{eq:decaynonl} we may derive an estimate for the solution to the scale-invariant damping Cauchy problem \eqref{eq:CP2}.
Coming back, by the inverse Liouville transformation, we find
\begin{equation*}
|v(t,|x|)|\le \<t\>^{-1}\<t+|x|\>^{-1}\<t-|x|\>^{-(\kappa-1)}.
\end{equation*}
The worst situation is close to the light cone, where we only have
\begin{equation*}
|v(t,|x|)|\le \<t\>^{-2}.
\end{equation*}
The decay behavior $\<t\>^{-2}$ in the $3$-dimensional case can be interpreted as $\<t\>^{-\frac{(n+2)-1}2}$: the same decay for the wave equation in
dimension $n+2$. This motivates the $2$-dimensions shift of the critical exponent $p_0(n) \to p_0(n+2)$.
\end{remark}

\section{Expectations for~$\mu\neq2$}

The same type of transformation we used in the treatment of~\eqref{eq:CP2} allows us to transform the Cauchy problem with scale-invariant mass and
dissipation
\begin{equation}\label{eq:CPgen}
\begin{cases}
v_{tt}-\triangle v + \frac\mu{1+t}\,v_t + \frac{m}{4(1+t)^2}\,v = |v|^p, & t\geq0, \quad x\in\R^n,\\
v(0,x)=v_0(x), & x\in\R^n,\\
v_t(0,x)=v_1(x), & x\in\R^n,
\end{cases}
\end{equation}
where~$\mu>0$ and~$m\in\R$, into
\begin{equation}\label{eq:CPgennew}
\begin{cases}
u_{tt}-\triangle u + \frac{\mu(2-\mu)+m}{4(1+t)^2} u= \<t\>^{-\frac\mu2(p-1)}|u|^p, & t\geq0, \quad x\in\R^n,\\
u(0,x)=u_0(x), & x\in\R^n,\\
u_t(0,x)=u_1(x), & x\in\R^n,
\end{cases}
\end{equation}
where we set~$u(t,x)=\<t\>^{\frac\mu2}v(t,x)$, $u_0=v_0$ and~$u_1=v_1+(\mu/2)v_0$.

In particular, in the special case~$m=(\mu-2)\mu$, the equation in~\eqref{eq:CPgennew} becomes a wave equation with the
nonlinearity~$\<t\>^{-\frac{\mu}{2}(p-1)}|u|^p$. We may directly follow the proof of Theorem~\ref{th:mu2} to obtain a nonexistence result for this special
problem.%\footnote{Marcello: I checked the subcritical case only, $p<\tilde p_\mu(n)$}
\begin{theorem}\label{th:mu}
Assume that~$v\in\mathcal{C}^2([0,T)\times \R^n)$ is a solution to~\eqref{eq:CPgen} with~$m=(\mu-2)\mu$ and initial data $(v_0,v_1)\in \mathcal C_c^2(\R^n)\times \mathcal C_c^1(\R^n)$ such that $v_1, v_0\geq0$, and $(v_0,v_1)\not \equiv (0,0)$. If~$p\in(1, \tilde p_\mu(n)]$, then~$T<\infty$, where
\[ \tilde p_\mu(n) = \max \left\{ p_\infty(n-1+\mu/2) , p_0(n+\mu) \right\}. \]
\end{theorem}
\begin{remark}
Let us compute~$\tilde p_\mu(n)$. The critical exponent~$p_0(n+\mu)$ may be written as
\[ p_0(n+\mu)= \frac{q+\sqrt{q^2+4(q-1)}}{2}, \qquad \text{where} \quad q=q(n+\mu)= 1 + \frac2{n-1+\mu} = p_\infty(n-1+\mu). \]
%
%In particular, $p_0(n+\mu)>p_\infty(n-1+\mu)$.
To determine~$\tilde p_\mu(n)$ we remark that~$p_0(n+\mu)\geq r\doteq p_\infty(n-1+\mu/2)$ if and only if
\[ \sqrt{q^2+4(q-1)} \geq 2r - q.  \]
Being~$q<r$ for any~$\mu>0$ we may take the squared powers:
\[ q^2+4(q-1) \geq 4r^2-4rq+q^2, \]
that is, $q-1 \geq r(r-q)$, explicitly,
\[ \frac{2}{n-1+\mu} \geq \frac{n+1+\mu/2}{n-1+\mu/2} \left(\frac{2}{n-1+\mu/2}-\frac2{n-1+\mu}\right) = \frac{\mu(n+1+\mu/2)}{(n-1+\mu/2)^2(n-1+\mu)}, \]
that is,
\[ \mu(n-3) + 2(n-1)^2 \geq 0. \]
It follows that~$\tilde{p}_\mu(n)=p_0(n+\mu)$ for any~$n\geq3$, $\tilde p_\mu(1)=p_\infty(\mu/2)$, and
\[ \tilde p_\mu(2) = \begin{cases}
p_\infty(1+\mu/2) & \text{if~$\mu\geq2$,}\\
p_0(2+\mu) & \text{if~$\mu\in[0,2]$.}
\end{cases} \]
\end{remark}
The statements of Theorem~\ref{th:mu} are consistent with known results for the classical semi-linear wave equation (i.e. $\mu=0$) and with Theorem~\ref{th:mu2} (i.e.
$\mu=2$), which are the only two cases in which~$m=0$.
\begin{proof}[Proof of Theorem~\ref{th:mu}]
We only sketch the differences to the proof of Theorem~\ref{th:mu2}. For the sake of brevity we only consider the subcritical case. It is clear that we
obtain
\begin{align*}
\ddot F(t)
    & \gtrsim \<t\>^{-(n+\mu/2)(p-1)}|F(t)|^o, \\
\ddot F(t)
    & \gtrsim \<t\>^{-(n-1+\mu/2)(p-1)+(n-1)\frac{p}2}\,|F_1(t)|^p.
\end{align*}
By virtue of Lemma~\ref{Lem:2YZ} we derive again~$F_1(t)>0$, and integrating twice the estimate for~$\ddot F(t)$, we derive~$F(t)\gtrsim \<t\>^a$, where
\[ a = \max \left\{-\frac{n-1+2\mu}2\,p+n+2, 1\right\}. \]
Setting~$q=(n+\mu/2)(p-1)$ we immediately obtain the blow-up in finite time if~$1>(q-2)/(p-1)$, i.e. $p<p_\infty(n-1+\mu/2)$, or if
\[ -\frac{n-1+2\mu}2\,p+n+2 > \frac{q-2}{p-1} = n+\frac\mu2 -\frac2{p-1}, \]
i.e.~$p<p_0(n+\mu)$.
\end{proof}
We may conjecture that global existence of small data solutions holds for some range of~$p>\tilde p_\mu(n)$. For this reason we propose as critical exponent $p_{crit}$
the value $\tilde p_\mu(n)$.
\begin{remark}
Let~$\mu\in(0,2)$ in~\eqref{eq:CP}. We may expect that the critical exponent $p_\mu(n)$ is not larger than~$\tilde p_\mu(n)$ due to the fact that the model in~\eqref{eq:CPgen}
with~$m=(\mu-2)\mu$ has an additional negative mass term with respect to the model in~\eqref{eq:CP}. Moreover, we know that the critical exponent has to be not smaller
than~$p_\infty(n-(1-\mu)_+)$. Therefore, we expect that
\[ p_\infty(n-(1-\mu)_+) \leq p_\mu(n) \leq \tilde p_\mu(n). \]
If~$n\geq2$, then we may replace~$\tilde p_\mu(n)=p_0(n+\mu)$, whereas if~$n=1$, we replace~$\tilde p_\mu(1)=p_\infty(\mu/2)$, i.e. we expect that
$p_\mu(1)\in [1+2/\mu,1+4/\mu]$. Indeed, if~$n=1$ our considerations are restricted to~$\mu\in(0,5/3)$, since we already know that the critical exponent
is~$3$ for~$\mu\geq 5/3$.

On the other hand, if~$n\geq3$ and $\mu\in(2,n+2)$ in~\eqref{eq:CP}, then we may expect that the critical exponent $p_\mu(n)$ is not smaller than~$\tilde p_\mu(n)$, due to the fact that the model in~\eqref{eq:CPgen} with~$m=(\mu-2)\mu$ has an additional positive mass term with respect to the model in~\eqref{eq:CP}. Moreover, we know that the critical exponent may not be smaller than~$p_\infty(n)$. Therefore, we expect that
\[ \max \left\{ p_0(n+\mu), p_\infty(n)\right\} \leq p_\mu(n) \leq p_0(n+2). \]
\end{remark}
\section{Concluding remarks and open problems}
\begin{remark}
In the statement of Theorem~\ref{Thm:ex} we may relax the assumption on the data from $(\bar v_0, \bar v_1)\in\mathcal{C}^2_c\times\mathcal{C}^1_c$
to~$(\bar v_0, \bar v_1)\in H^2\times H^1$, compactly supported.
\end{remark}
\begin{remark}
In the paper \cite{DAL14} the first two authors deal with the
odd dimensional cases $n \geq 5$. They prove the global existence of small data solutions to (\ref{eq:CP2}) for some admissible range of $p \in (p_0(n+2),p_1)$.
This yields together with the statement from Theorem \ref{th:mu2}, that $p_0(n+2)$ is the critical exponent for (\ref{eq:CP2}) in odd space dimensions $n \geq 5$, too.
It remains to analyze the case of even $n \geq 4$. But the authors expect the {\it shift $p_0(n) \to p_0(n+2)$ in all space dimensions $n \geq 4$}, too.
\end{remark}


\begin{thebibliography}{99}
\bibitem%[A]
{A}
{\sc F. Asakura}, {\it Existence of a global solution to a semi–linear wave equation with slowly decreasing initial data in three space dimensions.}
Comm. Partial Differential Equations {\bf 11}, (1986), 1459--1487.
\bibitem%[CH]
{CH}
{\sc R. Courant, D. Hilbert}, Methods of mathematical physics. Vol. 2. Interscience, New York 1962.
\bibitem{DA14+}
{\sc M. D'Abbicco}, \emph{The Threshold of Effective Damping for Semilinear Wave Equations}, Math. Methods Appl. Sci., to appear,
http://dx.doi.org/10.1002/mma.3126.
\bibitem%[ANS]
{DAL13} {\sc M. D'Abbicco, S. Lucente}, \emph{A modified test function method for damped wave equations}, Adv. Nonlinear Studies,
\bf 13 \rm (2013), 867--892.
\bibitem%[ANS]
{DAL14} {\sc M. D'Abbicco, S. Lucente}, \emph{NLWE with a special scale invariant damping in odd space dimension}, 8 pp., submitted for publication.
\bibitem
{DALR13} {\sc M. D'Abbicco, S. Lucente, M. Reissig}, \emph{Semilinear wave equations with effective damping}, Chinese Ann. Math., \textbf{34B} (2013) 3,
345--380, doi:10.1007/s11401-013-0773-0.
\bibitem%[Ga]
{Ga} {\sc A. Galstian}, \emph{Global existence for the one-dimensional second order semilinear hyperbolic equations},
J. Math. Anal. Appl., \bf 344 \rm (2008), 76--98.
\bibitem%[GLS]
{GLS} {\sc V. Georgiev, H. Lindblad,  C. D. Sogge}, \emph{Weighted Strichartz estimates and global existence for semilinear wave equations},
Amer. J. Math., \bf 119 \rm (1997), 1291--1319.
\bibitem%[G]
{G} {\sc R.T. Glassey}, \emph{Finite-time blow-up for solutions of nonlinear wave equations}, Math. Z., {\bf 177} (1981), 323--340.
\bibitem%[G2]
{G2} {\sc R.T. Glassey}, \emph{Existence in the large for $\Box u= F (u)$ in two space dimensions},  Math Z., {\bf 178} (1981), 233--261.
\bibitem
{ITY09} {\sc R. Ikehata, G. Todorova, B. Yordanov}, \emph{Critical exponent for semilinear wave equations with space-dependent potential},
Funkcial. Ekvac., {\bf 52} (2009), 411--435.
\bibitem
{ITY13} {\sc R. Ikehata, G. Todorova, B. Yordanov}, \emph{Optimal Decay Rate of the Energy for Wave Equations with Critical Potential}, J. Math. Soc. Japan
\textbf{65} (2013) 1, 183--236, doi:10.2969/jmsj/06510183.
\bibitem%[J]
{J}  {\sc F. John}, \emph{Blow-up of solutions of nonlinear wave equations in three space dimensions}, Manuscripta Math., {\bf 28} (1979), 235--268.
\bibitem%[JZ]
{JZ} {\sc H. Jiao, Z. Zhou}, \emph{An elementary proof of the blow-up for semilinear wave equation in high space dimensions},
Journal Differential Equations, \bf 189 \rm (2003), 355--365.
\bibitem%[K5]
{K5}
{\sc H. Kubo}, {\it Slowly decaying solutions for semilinear wave equations in odd space dimensions.}
Nonlinear Anal., {\bf 28}, (1997), 327--357.
\bibitem%[K]
{Kl} {\sc S. Klainerman}, \emph{Uniform decay estimates and Lorentz invariance of the classical wave equation}, Comm. Pure Appl. Math., \bf 38 \rm (1985), 321--332.
\bibitem%[LX]
{LX} {\sc T. Li, Y. Xin}, \emph{Life–span of classical solutions to fully nonlinear wave equations},
Comm. Partial Differential Equations, \bf 16 \rm (1991), 909--940.
\bibitem
{LNZ12} {\sc J. Lin, K. Nishihara, J. Zhai},
\emph{Critical exponent for the semilinear wave equation with time-dependent damping}, Discrete Contin. Dyn. Syst.,
\textbf{32} (2012) 12, 4307--4320, doi:10.3934/dcds.2012.32.4307.
\bibitem
{N11} {\sc K. Nishihara}, \emph{Asymptotic behavior of solutions to the semilinear wave equation with time-dependent damping}, Tokyo J. Math.,
\textbf{34} (2011), 327--343.
\bibitem%[Sc]
{Sc} {\sc J. Schaeffer}, \emph{The equation $u_{tt}-\Delta u=|u|^p$ for the critical value of $p$}, Proc. Roy. Soc. Edinburgh
Sect. A, \bf 101 \rm
(1985), 31--44.
\bibitem%[Si]
{Si}  {\sc T.C. Sideris}, \emph{Nonexistence of global solutions to semilinear wave equations in high dimensions},
J. Differential Equations, \bf 52 \rm (1984), 378--406.
\bibitem%[T]
{T}{\sc D. Tataru}, \emph{Strichartz estimates in the hyperbolic space and global existence for the semilinear wave equation},
Trans. Amer. Math. Soc., \bf 353 \rm (2001), 795--807.
\bibitem%[Y]
{Yag} {\sc K. Yagdjian}, \emph{Global existence in the Cauchy problem for nonlinear wave equations with variable speed of propagation},
New trends in the theory of hyperbolic equations, Birkh\"auser, Basel, 2005, 301--385.
\bibitem%[YZ]
{YZ06} {\sc B.T. Yordanov, Qi S. Zhang}, \emph{Finite time blow up for critical wave equations in high dimensions},
J. Func. Anal., \textbf{231} (2006), 361--374.
\bibitem
{W12} {\sc Y. Wakasugi}, \emph{Small data global existence for the semilinear wave equation with space-time dependent damping},
J. Math. Anal. Appl., \textbf{393} (2012) 66--79.
\bibitem%[W]
{W} {\sc Y. Wakasugi}, \emph{Critical exponent for the semilinear wave equation with scale invariant damping},
Fourier analysis (M. Ruzhansky and V. Turunen eds), Trends in Mathematics, Springer, Basel, 2014, 375--390.
\bibitem%[Wirth]
{Wirth} {\sc J. Wirth}, \emph{Solution representations for a wave equation with weak dissipation}, Math. Methods Appl. Sci., \textbf{27} (2004),
101--124.
\bibitem
{Z} {\sc Y. Zhou}, \emph{Cauchy problem for semilinear wave equations in four space dimensions with small initial data},
J. Differential Equations, \bf 8 \rm (1995), 135--144.
\end{thebibliography}
\end{document}